\documentclass[12pt, a4paper]{article}
\usepackage[utf8]{inputenc}
\usepackage[english]{babel}
\usepackage{amsmath,verbatim,amsthm,mathrsfs,stmaryrd}
\usepackage{comment}
\usepackage{indentfirst}
\usepackage{amstext}
\usepackage{enumerate}
\usepackage{amsfonts}
\usepackage{textcomp}
\usepackage{amssymb}
\usepackage[dvips]{graphicx}
\usepackage{setspace}
\usepackage[dvipsnames]{xcolor}
\usepackage{graphicx}
\graphicspath{ {./images/} }
\usepackage[all]{xy}
\usepackage[mathscr]{euscript}

\numberwithin{equation}{section}

\newcommand {\N}{\mathbb{N}}

\newcommand {\Z}{\mathbb{Z}}

\newcommand {\GG}{\mathcal{G}}

\newcommand {\G}{\mathcal{G}}
\newcommand {\F}{\mathbb{F}}

\newcommand{\lsp}{\operatorname{span}}

\newcommand {\al}{\alpha}

\newcommand{\p}{\mathfrak{p}}

\usepackage{geometry}
\newcommand{\scj}{\subseteq}
\newcommand{\up}{\mathfrak{p}}

\newcommand{\Lc}{\operatorname{Lc}}

\newcommand{\la}[1]{\text{$\mathcal{#1}$}}
\newcommand{\lb}[1]{\text{$\mathscr{#1}$}}

\newcommand{\powerset}[1]{\text{$P(#1)$}}								
\newcommand{\eword}{\text{$\omega$}}											

\newcommand{\xia}{\text{$\xi^\alpha$}}											

\newcommand{\uset}[1]{\text{$\uparrow\hspace{-0.1cm}{#1}$}}				
\newcommand{\usetr}[2]{\text{$\uparrow_{\scriptscriptstyle{#2}}\hspace{-0.1cm}{#1}$}} 
\newcommand{\lset}[1]{\text{$\downarrow\hspace{-0.1cm}{#1}$}}				
\newcommand{\lsetr}[2]{\text{$\downarrow_{\scriptscriptstyle{#2}}\hspace{-0.1cm}{#1}$}} 

\newcommand{\dgraphg}[1]{\text{$\lb{#1}$}}									
\newcommand{\dgraphupleg}[3]{\text{$\dgraphg{#1}=(\dgraphg{#1}^0,\dgraphg{#1}^1,#2,#3)$}}	

\newcommand{\dgraph}{\text{$\dgraphg{E}$}}									
\newcommand{\dgraphuple}{\text{$\dgraphupleg{E}{r}{s}$}}					
\newcommand{\alfg}[1]{\text{$\lb{#1}$}}										
\newcommand{\acfg}[1]{\text{$\lb{#1}$}}										
\newcommand{\lbfg}[1]{\text{$\lb{#1}$}}										
\newcommand{\acfrg}[1]{\text{$\acfg{B}_{#1}$}}								

\newcommand{\alf}{\text{$\alfg{A}$}}											
\newcommand{\acf}{\text{$\acfg{B}$}}											
\newcommand{\lbf}{\text{$\lbfg{L}$}}											
\newcommand{\acfra}{\text{$\acfg{B}_{\alpha}$}}								
\newcommand{\lgraphg}[2]{\text{$(\dgraphg{#1},\lbfg{#2})$}}				
\newcommand{\lspaceg}[3]{\text{$(\dgraphg{#1},\lbfg{#2},\acfg{#3})$}}

\newcommand{\lgraph}{\text{$\lgraphg{E}{L}$}}								
\newcommand{\lspace}{\text{$\lspaceg{E}{L}{B}$}}							

\newcommand{\awsetg}[2]{\text{$\lbfg{#1}^{#2}$}}	
\newcommand{\awplus}{\text{$\awsetg{L}{\scriptscriptstyle{\geq 1}}$}}								
\newcommand{\awn}[1]{\text{$\awsetg{L}{#1}$}}								
\newcommand{\awstar}{\text{$\awsetg{L}{\ast}$}}							
\newcommand{\awinf}{\text{$\awsetg{L}{\infty}$}}							
\newcommand{\awleinf}{\text{$\awsetg{L}{\scriptscriptstyle{\leq\infty}}$}}					

\newcommand{\fset}[1]{\text{$\mathsf{#1}$}}										
\newcommand{\filt}{\text{$\fset{F}$}}											

\newcommand{\ftight}{\text{$\fset{T}$}}										
\newcommand{\ftightw}[1]{\text{$\fset{T}_{#1}$}}

\newcommand{\ftg}[1]{\text{$\la{#1}$}}											
\newcommand{\ft}{\text{$\ftg{F}$}}												

\newtheorem{teorema}[equation]{Theorem}
\newtheorem{theorem}[equation]{Theorem}
\newtheorem{lema}[equation]{Lemma}

\newtheorem{corolario}[equation]{Corollary}

\newtheorem{definicao}[equation]{Definition}

\newtheorem{proposicao}[equation]{Proposition}
\newtheorem{proposition}[equation]{Proposition}
\newtheorem{exemplo}[equation]{Example}
\newtheorem{remark}[equation]{Remark}

\title{Ultragraph algebras via labelled graph groupoids, with applications to generalized uniqueness theorems}
\author{Gilles G. de Castro\footnote{Partially supported by Capes-PrInt Brazil grant number 88881.310538/2018-01.}, Daniel Gon\c{c}alves\footnote{Partially supported by Conselho Nacional de Desenvolvimento Cient\'ifico e Tecnol\'ogico (CNPq) grant numbers 304487/2017-1 and 406122/2018-0  and Capes-PrInt grant number 88881.310538/2018-01 - Brazil.}, and Daniel W. van Wyk}

\begin{document}

\maketitle

\begin{abstract}
An ultragraph gives rise to a labelled graph with some particular properties. In this paper we describe the algebras associated to such labelled graphs as groupoid algebras. More precisely, we show that the known groupoid algebra realization of ultragraph C*-algebras is only valid for ultragraphs for which the range of each edge is finite, and we extend this realization to any ultragraph (including ultragraphs with sinks). Using our machinery, we characterize the  shift space associated to an ultragraph as the tight spectrum of the inverse semigroup associated to the ultragraph (viewed as a labelled graph). Furthermore, in the purely algebraic setting, we show that the algebraic partial action used to describe an ultragraph Leavitt path algebra as a partial skew group ring is equivalent to the dual of a topological partial action, and we use this to describe  ultragraph Leavitt path algebras as Steinberg algebras. Finally, we prove generalized uniqueness theorems for both ultragraph C*-algebras and ultragraph Leavitt path algebras and characterize their abelian core subalgebras.

\end{abstract}

\vspace{0.5pc}

{\bf Keywords:} Ultragraphs, C*-algebras, Leavitt path algebras, Steinberg algebras, Partial skew group rings, Labelled spaces.

{\bf MSC2020:} 16S88, 16S35, 46L55, 47L55, 22A22.  

\newpage

\tableofcontents

\newpage
\section{Introduction}

Ultragraphs are versatile types of labelled graphs: their combinatorial structure encodes elaborate symbolic dynamics and algebras, meanwhile the intuition from the graph context is often (but not always) preserved. Ultragraphs were originally defined in \cite{MR2050134}, as an unifying object to study Exel-Laca and graph C*-algebras. Since then, their study has intertwined Dynamics, Algebra and Analysis in ways that each area benefits from the other, see for example \cite{MR3600124, SoboG, gonccalves2018li, gonccalves2019ultragraph, tasca2020kms} where ultragraphs are used to study shift spaces over infinite alphabets, \cite{MR3856223, tasca2020kms} where KMS states associated to ultragraph C*-algebras are described, \cite{Larki} where purely infinite ultragraph C*-algebras are determined, \cite{GDD} where topological full groups associated to ultragraph groupoids are shown to be isomorphism invariants, \cite{imanfar2017leavitt} where ultragraph Leavitt path algebras are introduced, \cite{GRirred} where irreducible representations of ultragraph Leavitt path algebras are characterized, among many other developments. 

Since ultragraphs are labelled graphs (for which edges with the same label necessarily have the same source), their algebras share some of the intricacies of labelled graph algebras. Labelled graph C*-algebras were originally defined in \cite{MR2304922}, but the final definition was settled independently in \cite{MR3614028} and \cite{MR3680957}. Technically, this C*-algebra is associated with a labelled space, which is a labelled graph with an additional family of sets of vertices (see Section \ref{subsect.labelled.spaces}). 
In this paper we illustrate that the elaborate nature of labelled graphs (spaces) and their associated algebras manifests itself in ultragraphs, and we show when relevant concepts can be simplified for certain cases. 

To associate an algebra with an ultragraph, we use the strategy proposed for combinatorial algebras in \cite{MR2419901}. Specifically, for an ultragraph $\G$ an inverse semigroup is associated with $\G$.  The tight spectrum of this inverse semigroup is then used via a groupoid, or partial crossed product, construction as a building block for the C*-algebra or the Leavitt path algebra associated with $\G$. In \cite{MR3648984} an inverse semigroup is associated with a labelled graph and the tight spectrum of the inverse semigroup is characterized. Since ultragraphs are special cases of labelled graphs, we use the inverse semigroup of a labelled graph to study general ultragraphs. 
In \cite{MR3938320} a shift space is associated with an ultragraph $\G$ without sinks that satisfies Condition~(RFUM), and also in \cite{tasca2020kms} if $\G$ has sinks and satisfies Condition~(RFUM2). For these special cases we show that shift spaces coincide with the tight spectrum of the inverse semigroup associated with $\G$ (see Theorem~\ref{thm:ftight.homeo.X}). 
However, we find that the ``graph'' like picture of a shift space (using paths and cylinder sets) associated to an ultragraph satisfying Condition~(RFUM2) is not available for general ultragraphs.

 
The tight spectrum of the inverse semigroup mentioned above is in fact the unit space of a groupoid used to realize an ultragraph C*-algebra as a groupoid C*-algebra. In \cite{MR2457327} an inverse semigroup is associated with an ultragraph to also form a groupoid such that the groupoid C*-algebra is isomorphic to the ultragraph C*-algebra.  However, we point out that our description of the unit space differs from that of \cite{MR2457327}, and we allow for sinks. Essentially, comparing these groupoids, the unit space in  \cite{MR2457327} lacks certain elements. We elaborate in detail on this in Remarks~\ref{rmk:filters.finite.type} and \ref{rmk:ultraset}, and in Remark~\ref{rmk:MM.groupoid} we discuss some implications to the associated groupoid C*-algebra. In order for the description in \cite{MR2457327} to work, an extra assumption that the range of every edge is finite is needed, see Remark~\ref{rmk:ultraset}. 

Often when a combinatorial object has a C*-algebra associated with it, one can also associate an algebra over a ring with it. If a groupoid description of this C*-algebra is obtained, then it is natural to try to obtain a Steinberg algebra realization of the algebra. Since the theory of Steinberg algebras is quite developed (see \cite{CEM, center, Ben} for a few examples), and closely parallels the theory of groupoid C*-algebras, such realizations open 
the possibility of many advances. In the case of graphs, the groupoid used for modelling Leavitt path algebras and graph C*-algebras is the same. We show that this is also the case for ultragraphs\footnote{While we were in the final organizing phase of this paper, the manuscript \cite{HazNam} was posted on arXiv, with a Steinberg algebra realization of ultragraph Leavitt path algebras. Nevertheless the groupoip presented there is the same as the groupoid introduced in \cite{MR2457327}, and so their approach needs the extra assumption that the range of each edge is finite.},  without any assumptions on the ultragraph. 

To obtain a Steinberg algebra realization of an ultragraph Leavitt path algebra, we use the description of the latter as partial skew group ring \cite{goncalves_royer_2019} and the results of \cite{MR3743184}. With this approach we develop some new tools along the way which may be of independent interest. However, we need to overcome a few hurdles. First, we need to reconcile the two running definitions of ultragraph Leavitt path algebras. To describe this issue, recall that an ultragraph is a quadruple $\mathcal{G}=(G^0, \mathcal{G}^1, r,s)$ consisting of two countable sets $G^0$ and $\mathcal{G}^1$, a map $s:\mathcal{G}^1 \to G^0$, and a map $r:\mathcal{G}^1 \to P(G^0)\setminus \{\emptyset\}$, where $P(G^0)$ is the power set of $G^0$. A key concept in the definition of ultragraph algebras is that of a generalized vertex. Originally in \cite{MR2050134} the set of generalized vertices
of an ultragraph $ \mathcal{G}$ is defined as the smallest subset of $P(G^0)$ that contains $\{v\}$ for all $v\in G^0$, contains $r(e)$ for all $e\in \mathcal{G}^1$, and is closed under finite unions and nonempty finite intersections. 
However, another definition also appears in the literature. For example, in \cite{MR2413313} the set of generalized vertices, in addition to the original definition, is also taken to be closed under relative complements. However, it is proved that the C*-algebras arising from both definitions are isomorphic (see Remark~2.5 and Section 3 of \cite{MR2413313}). In the algebraic setting, ultragraph Leavitt path algebras are defined in \cite{imanfar2017leavitt} by also including relative complements in the generalized vertices (this convention is also used in \cite{Firrisa}). In other contexts the definition of generalized vertices does not allow for the relative complements (see \cite{GRirred, reduction, goncalves_royer_2019, HazNam, Nam}).  
Analogous to the C*-algebraic case, we show that ultragraph Leavitt path algebras arising from both definitions agree, see Proposition~\ref{prop:isom.leavitt.path.algs}. Then we have the realization of any ultragraph Leavitt path algebra as a partial skew group ring, as given in \cite{goncalves_royer_2019}, at our disposal. This bring us to the second hurdle to overcome. We show that the purely algebraic partial action constructed in \cite{goncalves_royer_2019} from an ultragraph can be built from the topological partial action defined on the tight spectrum of the inverse semigroup associated with the ultragraph in \cite{MR4109095}, see Theorem~\ref{prop:equivalent.partial.actions}. We can then build the transformation groupoid of the topological partial action and use the isomorphism between the Steinberg algebra and the partial skew group ring found in \cite{MR3743184} (see Theorem~\ref{kart}).

To exemplify the power of realizing ultragraph algebras as groupoid algebras we apply our results to obtain generalized uniqueness theorems to both ultragraph C*-algebras and ultragraph Leavitt path algebras. Recall that a uniqueness theorem 
means a set of conditions on a combinatorial object (in our case, an ultragraph) or on a representation of the associated algebra, guaranteeing that the representation is injective (see \cite{saragab} for a nice description of the development of the subject in the context of C*-algebras). 
In our context, after identifying the abelian core subalgebra of an ultragraph algebra, we apply general groupoid C*-algebras and Steinberg algebra results \cite{BNRSW,CEP} to show that a representation of an ultragraph algebra is injective if, and only if, its restriction to the abelian core subalgebra is injective, see Theorem~\ref{pizza} and Theorem~\ref{uniqueultra}. We remark that our results do not require aperiodicity, a gauge invariance or a $\Z$-graded assumption. Finally, our identification of the abelian core subalgebra of an ultragraph Leavitt path algebra allows us to apply results of \cite{HazratLi} to describe the core as the centralizer of the diagonal subalgebra, and to show that if the center of an ultragraph Leavitt path algebra is equal to its abelian core then the ultragraph is either a single vertex or a vertex with a loop (in the context of Leavitt path algebras this question is posed in \cite{GILCANTO2018227} and answered in \cite{HazratLi}).

We organized the paper in a manner that (we hope) appeals to an audience with different backgrounds and interests. In Section~\ref{subsection:filters.and.characters} we recall the necessary concepts regarding filters, labelled spaces, the key inverse semigroup associated to a labelled space and its tight spectrum. In Section~\ref{groupoids} we recall the definition of the shift space associated to an ultragraph that satisfies Condition~(RFUM2) and, in Theorem~\ref{thm:ftight.homeo.X}, we show that it is homeomorphic to the tight spectrum of the associated inverse semigroup. As a consequence we obtain that the ``graph like'' groupoids studied in \cite{GDD, tasca2020kms}  coincide with the groupoid arising from the labelled space perspective. Furthermore, we provide examples of ultragraphs for which the tight spectrum of the inverse semigroup associated to it cannot be described in the same manner as the shift defined for ultragraphs with Condition~(RFUM2). Section~4 is independent of Section~\ref{groupoids}. Here we show that the algebraic partial action defined in \cite{goncalves_royer_2019} agrees with the topological partial action associated to an ultragraph defined in \cite{MR4109095}. We devote Section~5 to reconciling the different definitions of an ultragraph Leavitt path algebra (Proposition~\ref{prop:isom.leavitt.path.algs}) and to the realization of ultragraph algebras as groupoid algebras (C*-algebraic - Theorem~\ref{thm:isom.cstar.algs}, purely algebraic - Theorem~\ref{kart}) using the groupoid described in Section~\ref{groupoids} and, in the algebraic case, the results of Section~4. Finally, in Section~6, we apply our results to obtain generalized uniqueness theorems for ultragraph algebras, both in the analytical (Theorem~\ref{pizza}) and algebraic context (Theorem~\ref{uniqueultra}), and describe the abelian core subalgebra of an ultragraph algebra (Proposition~\ref{viajar}).

\section{Preliminaries}

Throughout this paper we let $\N$ and $\N^*$ denote the set of non-negative integers, and positive integers, respectively.


\subsection{Filters}\label{subsection:filters.and.characters}

A \emph{filter} in a partially ordered set $P$ with least element $0$ is a subset $\xi$ of $P$ such that (i) $0\notin\xi$; (ii) if $x\in\xi$ and $x\leq y$, then $y\in\xi$ and (iii) if $x,y\in\xi$, there exists $z\in\xi$ such that $z\leq x$ and $z\leq y$. If $P$ is a (meet) semilattice, condition (iii) may be replaced by $x\wedge y\in\xi$ if $x,y\in\xi$. An \emph{ultrafilter} is a filter which is not properly contained in any filter. 

For  $x\in P$, we define \[\uset{x}=\{y\in P \ | \ x\leq y\}\ \ ,\ \  \lset{x}=\{y\in P \ | \ y\leq x\},\] and for subsets $X,Y$ of $P$ define \[\uset{X} = \bigcup_{x\in X}\uset{x} = \{y\in P \ | \ x\leq y \ \mbox{for some} \ x\in X\},\]
and $\usetr{X}{Y} = Y\cap\uset{X}$; the sets $\usetr{x}{Y}$, $\lsetr{x}{Y}$, $\lset{X}$ and $\lsetr{X}{Y}$ are defined analogously. A filter $\xi$ is called \emph{principal} if $\xi=\uset{x}$ for some $x\in P$.




If $\xi$ is a filter in a lattice $L$ with least element $0$, we say that $\xi$ is \emph{prime} if for every $x,y\in L$, if $x\vee y\in \xi$, then $x\in\xi$ or $y\in\xi$.

We  consider a \emph{Boolean algebra} to be a relatively complemented distributive lattice with least element 0. We do not assume that Boolean algebras have a greatest element. 

The following result is well known in order theory.

\begin{proposition}\label{prop:ultrafilter}
	Let $\xi$ be a filter in a Boolean algebra $\acf$. Then the following are equivalent:
	\begin{itemize}
	    \item $\xi$ is an ultrafilter,
	    \item $\xi$ is a prime filter,
	    \item if $x\in\acf$ is such that $x\wedge y\neq 0$ for all $y\in\xi$, then $x\in\xi$.
	\end{itemize}
	
\end{proposition}

We briefly describe the Stone duality. For a Boolean algebra $\acf$, the set of all ultrafilters in $\acf$ will be denoted by $\widehat{\acf}$. For each $x\in \acf$, we let $U_x=\{\xi\in\widehat{\acf}\mid x\in\xi\}$. Then the family $\{U_x\}_{x\in\xi}$ is a basis of compact-open sets for a Hausdorff topology on $\widehat{\acf}$. The set $\widehat{\acf}$ with this topology is called the Stone dual of $\acf$. On the other hand, if $X$ is a Hausdorff space such the set of all compact-open sets $\acfg{K}(X)$ is a basis for the topology on $X$, then $\acfg{K}(X)$ is a Boolean algebra such that $\widehat{\acfg{K}(X)}$ is homeomorphic to $X$.

\subsection{Algebras of sets via characteristic functions}

Let $\acfg{C}$ be a lattice of subsets of a set $X$. In this subsection, we characterize the algebra generated by the characteristic functions (taking values in a commutative unital ring) of elements from $\acfg{C}$ as an universal algebra. This will be important in Section~\ref{repolho}, in the study of ultragraph Leavitt path algebras. However, since this result is independent from ultragraph algebras, we present it here.

Let $X$ be a set, $R$ a commutative unital ring, $\acfg{C}\subset\powerset{X}$ be a family of subsets of $X$, $F(X)$ be the $R$-algebra of functions from $X$ to $R$ with pointwise operations, $F_{\acfg{C}}(X)$ be the subalgebra generated by $\{1_C\}_{C\in\acfg{C}}$, where $1_C$ represents the characteristic function of the set $C$, and let $\acfg{A}_{\acfg{C}}$ be the algebra of sets generated by $\acfg{C}$.

\begin{lema}\label{lem:boolean.algebra.function}
We have that
\[\acfg{A}_{\acfg{C}}=\{A\scj X\mid 1_A\in F_{\acfg{C}}(X)\}.\]
In particular $F_{\acfg{C}}(X)= F_{\acfg{A}_{\acfg{C}}}(X)$.
\end{lema}

\begin{proof}
Define $\acf=\{A\scj X\mid 1_A\in F_{\acfg{C}}(X)\}$. It is straightforward to check that $\acf$ is an algebra of sets containing $\acfg{C}$ so that $\acfg{A}_{\acfg{C}}\scj \acf$. Let $B\in \acf$ so that $1_B$ is a linear combination of some $1_{C_1}\,\ldots,1_{C_n}$ where $C_i$ is a finite intersection of elements of $\acfg{C}$ for $i=1,\ldots,n$. By basic set theory, we can find a family of mutually disjoint sets  $A_1,\ldots,A_m$ in $\acfg{A}_{\acfg{C}}$ such that each $C_i$ is a finite union of some elements of this family. We can then rewrite $1_B$ as a linear combination of $1_{A_1},\ldots,1_{A_m}$ where each coefficient is either 0 or 1. If follows that $B$ is the union of the $A_j$'s whose coefficient is 1 and hence $B\in\acfg{A}_{\acfg{C}}$.

For the last part, the inclusion $F_{\acfg{C}}(X)\scj F_{\acfg{A}_{\acfg{C}}}(X)$ is immediate and the reverse inclusion follows from the argument above.
\end{proof}

\begin{definicao}\label{def:univ.algebra.lattice.sets}
Let $\acfg{C}$ be a lattice of sets. We define $L_{\acfg{C}}$ as the universal $R$-algebra generated by a family of idempotents $\{p_A\}_{A\in \acfg{C}}$ subject to the relations $p_{\emptyset}=0$, $p_{A\cap B}=p_{A}p_{B}$ and $p_{A\cup B}=p_A+p_B-p_{A\cap B}$ for all $A,B\in \acfg{C}$.
\end{definicao}





\begin{proposicao}\label{prop:univ.algebra.lattice.sets}
Let $\acfg{C}$ be a lattice of subsets of $X$. Then there exists an isomorphism $\phi:L_{\acfg{C}}\to F_{\acfg{C}}(X)$ such that $\phi(p_A)=1_A$ for all $A\in \acfg{C}$.
\end{proposicao}

\begin{proof}
By the universal property of $L_{\acfg{C}}$, there exists a $R$-algebra homomorphism $\phi:L_{\acfg{C}}\to F(X)$ such that $\phi(p_A)=1_A$ for all $A\in\acfg{C}$. Notice that $\phi(L_{\acfg{C}})=F_{\acfg{C}}(X)$, so it remains to show that $\phi$ is injective. We claim that for $A,B\in\acfg{C}$, we have that $p_A=p_{A\cap B}$ if and only if $1_A=1_{A\cap B}$. Indeed, on the one hand, if $p_A=p_{A\cap B}$ then $1_A=\phi(p_A)=\phi(p_{A\cap B})=1_{A\cap B}$. On the other hand, if $1_A=1_{A\cap B}$ then $A=A\cap B$ so that $p_A=p_{A\cap B}$.

Take $x\in L_{\acfg{C}}$ and suppose that $\phi(x)=0$. Since $\acfg{C}$ is closed under intersections, we have that $L_{\acfg{C}}=\lsp \{p_A\}_{A\in\acfg{C}}$, so we can write $x=\sum_{i=1}^n r_ip_{A_i}$ for some $n\in\N^*$ and $r_i\in R$, $A_i\in\acfg{C}$ for all $i=1,\ldots,n$. For each $\emptyset\neq I\scj\{1,\ldots,n\}$, we let $A_{I}=\bigcap_{i\in I}A_i$, $B_{I}=\bigcup_{i\in I^c}A_i$ and $p_I=p_{A_I}-p_{A_I\cap B_I}$. Straightforward computations show that $p_I$ is idempotent for all $I$, that $p_Ip_J=0$ if $I\neq J$ and that $p_{A_i}=\sum_{i\in I}p_I$ for all $i$. By the above claim $p_I=0$ if and only if $A_I=A_I\cap B_I$. Let $\Lambda=\{I\scj\{1,\ldots,n\}\mid p_I\neq 0\}$ so that we can write $x=\sum_{I\in \Lambda}s_Ip_I$, where $s_I\in R$ for all $I\in\Lambda$. For a fixed $J\in\Lambda$, we have that
\[0=\phi(x)=\phi(p_J)\phi(x)=\phi(p_Jx)=\phi(s_Jp_J)=s_J(1_{A_J}-1_{A_J\cap B_J})=s_J 1_{A_J\setminus B_J}.\]
Since $J\in\Lambda$, we have that $p_J\neq 0$, so that $A_J\neq A_J\cap B_J$, and hence $A_J\setminus B_J\neq\emptyset$. The above equality then implies $s_J=0$ and since $J\in\Lambda$ was arbitrary, we have that $x=0$. The injectivity of $\phi$ follows.
\end{proof}

\subsection{Labelled spaces and ultragraphs}\label{subsect.labelled.spaces}

A \emph{(directed) graph} $\dgraphuple$ consists of non-empty sets $\dgraph^0$ (of \emph{vertices}), $\dgraph^1$ (of \emph{edges}), and \emph{range} and \emph{source} functions $r,s:\dgraph^1\to \dgraph^0$. A vertex $v$ such that $s^{-1}(v)=\emptyset$ is called a \emph{sink}, and the set of all sinks is denoted by $\dgraph^0_{sink}$. The graph is countable if both $\dgraph^0$ and $\dgraph^1$ are countable.

A \emph{path of length $n$} on a graph $\dgraph$ is a sequence $\lambda=\lambda_1\lambda_2\ldots\lambda_n$ of edges such that $r(\lambda_i)=s(\lambda_{i+1})$ for all $i=1,\ldots,n-1$. We write $|\lambda|=n$ for the length of $\lambda$ and regard vertices as paths of length $0$. $\dgraph^n$ stands for the set of all paths of length $n$ and $\dgraph^{\ast}=\cup_{n\geq 0}\dgraph^n$. Similarly, we define a \emph{path of infinite length} (or an \emph{infinite path}) as an infinite sequence $\lambda=\lambda_1\lambda_2\ldots$ of edges such that $r(\lambda_i)=s(\lambda_{i+1})$ for all $i\geq 1$; for such a path, we write $|\lambda|=\infty$ and we let $\dgraph^{\infty}$ denote the set of all infinite paths.

A \emph{labelled graph} consists of a graph $\dgraph$ together with a surjective \emph{labelling map} $\lbf:\dgraph^1\to\alf$, where $\alf$ is a fixed non-empty set, called an \emph{alphabet}, and whose elements are called \emph{letters}. $\alf^{\ast}$ stands for the set of all finite \emph{words} over $\alf$, together with the \emph{empty word} \eword, and $\alf^{\infty}$ is the set of all infinite words over $\alf$. 	We consider $\alf^{\ast}$ as a monoid with operation given by concatenation. In particular, given $\alpha\in \alf^{\ast}\setminus\{\eword\}$ and $n\in\N^*$, $\alpha^n$ represents $\alpha$ concatenated $n$ times and $\alpha^{\infty}\in \alf^{\infty}$ is $\alpha$ concatenated infinitely many times.

The labelling map $\lbf$ extends in the obvious way to $\lbf:\dgraph^n\to\alf^{\ast}$ and $\lbf:\dgraph^{\infty}\to\alf^{\infty}$. $\awn{n}=\lbf(\dgraph^n)$ is the set of \emph{labelled paths $\alpha$ of length $|\alpha|=n$}, and $\awinf=\lbf(\dgraph^{\infty})$ is the set of \emph{infinite labelled paths}. We consider $\eword$ as a labelled path with $|\eword|=0$, and set $\awplus=\cup_{n\geq 1}\awn{n}$, $\awstar=\{\eword\}\cup\awplus$, and $\awleinf=\awstar\cup\awinf$.  

For $\alpha\in\awstar$ and $A\in\powerset{\dgraph^0}$ (the power set of $\dgraph^0$), the \emph{relative range of $\alpha$ with respect to} $A$ is the set
\[r(A,\alpha)=
\begin{cases}
\{r(\lambda)\ |\ \lambda\in\dgraph^{\ast},\ \lbf(\lambda)=\alpha,\ s(\lambda)\in A\}, &
\text{if }\alpha\in\awplus \\ 
A, & \text{if }\alpha=\eword.
\end{cases}
\]
The \emph{range of $\alpha$}, denoted by $r(\alpha)$, is the set \[r(\alpha)=r(\dgraph^0,\alpha),\]
so that $r(\eword)=\dgraph^0$ and, if $\alpha\in\awplus$, then  $r(\alpha)=\{r(\lambda)\in\dgraph^0\ |\ \lbf(\lambda)=\alpha\}$.

We also define \[\lbf(A\dgraph^1)=\{\lbf(e)\ |\ e\in\dgraph^1\ \mbox{and}\ s(e)\in A\}=\{a\in\alf\ |\ r(A,a)\neq\emptyset\}.\]

A labelled path $\alpha$ is a \emph{beginning} of a labelled path $\beta$ if $\beta=\alpha\beta'$ for some labelled path $\beta'$. Labelled paths  $\alpha$ and $\beta$ are \emph{comparable} if either one is a beginning of the other. If $1\leq i\leq j\leq |\alpha|$, let $\alpha_{i,j}=\alpha_i\alpha_{i+1}\ldots\alpha_{j}$ if $j<\infty$ and $\alpha_{i,j}=\alpha_i\alpha_{i+1}\ldots$ if $j=\infty$. If $j<i$ set $\alpha_{i,j}=\eword$. Define $\overline{\awinf}=\overline{\lbf(\dgraph^{\infty})}=\{\alpha\in\alf^{\infty}\mid \alpha_{1,n}\in\awstar,\forall n\in\N\}$, that is, it is the set of all infinite words such that all beginnings are finite labelled paths. 
Also we write $\overline{\awleinf}=\awstar\cup\overline{\awinf}$.

A \emph{labelled space} is a triple $\lspace$ where $\lgraph$ is a labelled graph and $\acf$ is a family of subsets of $\dgraph^0$ which is closed under finite intersections and finite unions, contains all $r(\alpha)$ for $\alpha\in\awplus$, and is \emph{closed under relative ranges}, that is, $r(A,\alpha)\in\acf$ for all $A\in\acf$ and all $\alpha\in\awstar$. A labelled space $\lspace$ is \emph{weakly left-resolving} if for all $A,B\in\acf$ and all $\alpha\in\awplus$ we have $r(A\cap B,\alpha)=r(A,\alpha)\cap r(B,\alpha)$. A weakly left-resolving labelled space such that $\acf$ is closed under relative complements will be called \emph{normal}\footnote{
        Note this definition differs from \cite{MR3614028}. However since all labelled spaces             considered in this paper are weakly left-resolving, we include `weakly-left resolving' in the
        definition of a  normal labelled space.}. 

For $\alpha\in\awstar$, define \[\acfra=\acf\cap\powerset{r(\alpha)}=\{A\in\acf\mid A\scj r(\alpha)\}.\] If a labelled space is normal, then  $\acfra$ is a Boolean algebra for each $\alpha\in\awstar$. 

Our main objects of study in this paper are ultragraphs. Our approach is through labelled spaces, which will allow us to prove results for ultragraphs in full generality. Next, we outline how a labelled space is associated with an ultragraph, we highlight certain properties and we fix some notation.

\begin{definicao}\label{def of ultragraph}
An \emph{ultragraph} is a quadruple $\mathcal{G}=(G^0, \mathcal{G}^1, r,s)$ consisting of two countable sets $G^0, \mathcal{G}^1$, a map $s:\mathcal{G}^1 \to G^0$, and a map $r:\mathcal{G}^1 \to \powerset{G^0}\setminus \{\emptyset\}$, where $\powerset{G^0}$ is the power set of $G^0$.
\end{definicao}

\begin{definicao}\label{prop_vert_gen}
Let $\mathcal{G}$ be an ultragraph. Define $\mathcal{G}^0$ to be the smallest subset of $\powerset{G^0}$ that contains $\{v\}$ for all $v\in G^0$, contains $r(e)$ for all $e\in \mathcal{G}^1$, contains $\emptyset$, and is closed under finite unions and finite intersections. Elements of $\G^0$ are called \emph{generalized vertices}. The accommodating family $\acf$ associated to $\G$ is the smallest family of subsets of $G^0$ that contains $\G^0$ and is closed under relative complements, finite unions and finite intersections.
\end{definicao}

We can now define the labelled space associated to an ultragraph.

\begin{definicao}\label{labelfromultra} (\cite[Examples 3.3(ii) and 4.3(ii)]{MR2304922}, \cite[Example 2]{MR3614028}).
Fix an ultragraph $\mathcal{G}=(G^0, \mathcal{G}^1, r,s)$. Let $\dgraph_\mathcal{G}=(\dgraph^0_\mathcal{G}, \dgraph^1_\mathcal{G},r^{\prime},s^{\prime})$, where $\dgraph^0_\mathcal{G}=G^0$, $\dgraph^1_\mathcal{G}=\{(e,w):e\in\mathcal{G}^1, w\in r(e)$ and define  $r^{\prime}(e,w)=w$ and $s^{\prime}(e,w)=s(e)$. Set $\alf=\dgraph^1$, $\acf$ the accommodating family of $\G$ and define $\awn{}_\mathcal{G}:\dgraph^1\to\alf$ by $\awn{}_\mathcal{G}(e,w)=e$. Then,
$(\dgraph_\mathcal{G},\awn{}_\mathcal{G},\acf)$ is the normal labelled space associated to $\G$. 
\end{definicao}

Although ultragraphs are labelled graphs, it is important to recognize the usual notation used in their study. We briefly recall this below. For more details we refer the reader to \cite{MR2457327, tasca2020kms}.
 
Let $\G$ be an ultragraph. A \emph{finite path} in $\G$ of is either an element of $\G^{0}$ or a sequence of edges $e_{1}\ldots e_{k}$ in $\G^{1}$ where $s\left(e_{i+1}\right)\in r\left(e_{i}\right)$ for $1\leq i<k$. The set of finite paths in $\G$ is denoted by $\G^{\ast}$. 
An \emph{infinite path} in $\G$ is an infinite sequence of edges $\gamma=e_1e_2\ldots$ in $\prod \G^1$, where $s\left(e_{i+1}\right)\in r\left(e_i\right)$ for all $i\geq 1$.  The set of infinite paths in $\G$ is denoted by $\p^\infty$. With this definition, labelled paths correspond to paths on the ultragraph, and $\awinf=\overline{\awinf}$. Also, for any $A\in\acf$ and any $\alpha\in\awplus$ such that $s(\alpha)\in A$. A vertex $v$ in $G^0$ is called a \emph{sink} if $\left|s^{-1}\left(v\right)\right|=0$. We denote the set of sinks in an ultragraph $\G$ by $G_s^0$. We say that the ultragraph is \emph{connected} if for every $v,w\in G^0$, there exists a finite path $\gamma\in\G^*$ such that $s(\gamma)=v$ and $w\in r(\gamma)$.
For $n\geq1,$ we define $$\p^n:=\{(\alpha,A):\alpha\in\G^*,|\alpha |=n,A\in\G^0,A\subseteq r(\alpha)\}.$$ 
We specify that $\left(\alpha,A\right)=(\beta,B)$ if, and only if,
$\alpha=\beta$ and $A=B$. Letting $\p^0:=\G^0$ we define the \label{def_esp_ult} \emph{ultrapath space} associated with the ultragraph $\G$ to be $\p:=\displaystyle\bigsqcup_{n\geq 0}\p^n$.  The elements of $\p$ are called \emph{ultrapaths}. We embed the set of finite paths $\G^*$ in $\p$ by sending $\alpha$ to $(\alpha,r(\alpha))$. Define the length $\left\vert\left(\alpha,A\right)\right\vert$ of a pair $\left(\alpha,A\right)$ to be $\left\vert\alpha\right\vert$. Each $A\in\G^{0}$ is regarded as an ultrapath of length zero and can be identified with the pair $(A,A)$. Hence, the range map $r$ and the source map $s$ extends to $\p$ by declaring that $r\left(\left(\alpha,A\right)\right)=A$, $s\left(\left(\alpha,A\right)\right)=s\left(\alpha\right)$ and $r(A,A)=r(A)=A=s(A)=s(A,A)$.

A finite path $\al\in \GG^*$ with $|\al|>0$ is a \textit{loop} if $s(\al)\in r(\al)$. We say $\al$ is a loop based at $A\in \GG^0$ if $s(\al)\in A$. 
An \textit{exit} for a loop $\al=\al_1\ldots\al_n$ is either of the following:
\begin{enumerate}
    \item an edge $e\in \GG^1$ such that there exists an $i$ for which $s(e)\in r(\al_i)$,  but $e\neq  \al_{i+1}$,
    \item a sink $w$ such that $w \in r(\al_i)$ for some $i$.
\end{enumerate}

We concatenate elements in $\p$ in the following way: If $x=(\alpha,A)$ and $y=(\beta,B)$, with $|x|\geq 1, |y|\geq 1$, then $x\cdot y$ is defined if, and only if, $s(\beta)\in A$, in which case, $x\cdot y:=(\alpha\beta,B)$. 
Also we specify that:
	$$x\cdot y=
	\left\{
	\begin{array}{ll}
	x\cap y, & \text{if $x,y\in \G^0$ and if $x\cap y\neq\emptyset$}\\
	y, & \text{if $x\in\G^0$, $|y|\geq 1$, and if $s(y)\in x$}\\
	x_y, & \text{if $y\in\G^0$, $|x|\geq 1$, and if $r(x)\cap y\neq \emptyset$},
	\end{array}
	\right.$$  
where, if $x=\left(\alpha,A\right)$, $\left|\alpha\right|\geq1$ and if $y\in\G^{0}$, the expression $x_{y}$ is defined to be $\left(\alpha,A\cap y\right)$. 
We concatenate ultrapaths in $\p$ with paths in $\p^\infty$ as follows: If $s(\beta)\in r(\alpha,A)=A$, we define $(\alpha,A)\cdot\beta=\alpha\beta\in \p^\infty$, where $(\alpha,A)\in \p$ and $\beta \in \p^{\infty}$. Furthermore, if $\alpha=A$ then $(A,A)\cdot\beta=\beta\in\p^\infty$. 

\begin{remark}\label{rmk:no.relative.complement}
As pointed out in the introduction, the definition of $\G^0$ varies in the literature by either requiring that $\G^0$ is closed under relative complements or not. In this paper, \textbf{we do not require that $\G^0$ is closed under relative complements}.  This is consistent with Tomforde's original definition in \cite{MR2050134}. We shall keep track of this difference whenever it is relevant.
\end{remark}



\subsection{The inverse semigroup of a labelled space}\label{subsection:inverse.semigroup}

Let $\lspace$ be normal labelled space and consider the set \[S=\{(\alpha,A,\beta)\ |\ \alpha,\beta\in\awstar\ \mbox{and}\ A\in\acfrg{\alpha}\cap\acfrg{\beta}\ \mbox{with}\ A\neq\emptyset\}\cup\{0\}.\]
Define a  binary operation on $S$  as follows: $s\cdot 0= 0\cdot s=0$ for all $s\in S$ and, if $s=(\alpha,A,\beta)$ and $t=(\gamma,B,\delta)$ are in $S$, then \[s\cdot t=\left\{\begin{array}{ll}
(\alpha\gamma ',r(A,\gamma ')\cap B,\delta), & \mbox{if}\ \  \gamma=\beta\gamma '\ \mbox{and}\ r(A,\gamma ')\cap B\neq\emptyset,\\
(\alpha,A\cap r(B,\beta '),\delta\beta '), & \mbox{if}\ \  \beta=\gamma\beta '\ \mbox{and}\ A\cap r(B,\beta ')\neq\emptyset,\\
0, & \mbox{otherwise}.
\end{array}\right. \]
For  $s=(\alpha,A,\beta)\in S$, we define $s^*=(\beta,A,\alpha)$. 

The set $S$ endowed with the operations above is an inverse semigroup with zero element $0$ (\cite{MR3648984}, Proposition 3.4), whose semilattice of idempotents is \[E(S)=\{(\alpha, A, \alpha) \ | \ \alpha\in\awstar \ \mbox{and} \ A\in\acfra\}\cup\{0\}.\]

\begin{remark}
In \cite{MR2457327}, an inverse semigroup is associated with an ultragraph. The inverse semigroup we obtain from the labelled space of an ultragraph differs from the inverse semigroup in \cite{MR2457327} in two aspects. Firstly, \cite{MR2457327} intrinsically considers quadruples instead of triples and secondly, it uses $\G^0$, as opposed to $\acf$.
\end{remark}

The natural order in the semilattice $E(S)$ is given by $p\leq q$ if only if $pq=p$. In our case, the order can described by noticing that if $p=(\alpha, A, \alpha)$ and $q=(\beta, B, \beta)$, then $p\leq q$ if and only if $\alpha=\beta\alpha'$ and $A\subseteq r(B,\alpha')$ (see \cite[ Proposition 4.1]{MR3648984}).

\subsection{Filters in $E(S)$}\label{subsection:filters.E(S)}

For  a (meet) semilattice $E$ with $0$, there is a bijection between the set of filters in $E$ and the set $\hat{E}_0$ of characters of $E$ (that is, the zero and meet-preserving non-zero maps from $E$ into the Boolean algebra $\{0,1\}$). With the topology of pointwise convergence on $\hat{E}_0$, the closure of the subset $\hat{E}_\infty$ of characters that correspond to ultrafilters in $E$ is denoted by $\hat{E}_{tight}$, and is called the \emph{tight spectrum} of $E$. Elements of $\hat{E}_{tight}$ are the \emph{tight characters of $E$}, and their corresponding filters in $E$ are \emph{tight filters}. The set of all filters will be denoted by $\filt$ and the set of all 
tight filters will be denoted by $\ftight$, which we also call \emph{tight spectrum}. In particular, we may view $\ftight$ as a closed subspace of $\filt$. See \cite[Section 12]{MR2419901} for details. We use this construction for the semilattice $E(S)$ as in the previous section.

We now describe a basis for the topology of $\ftight$ inherited by the bijection with $\hat{E}_{tight}$ as done in \cite[Section 2.2]{MR2974110}. For each $e\in E(S)$, define
\begin{equation*} \label{dfn:tight.spec.basic.open.neigh}
    V_e=\{\xi\in\ftight\mid e\in\xi\}.
\end{equation*}
If  $\{e_1,\ldots,e_n\}$ is a finite (possibly empty) set in $E(S)$, define
\[ V_{e:e_1,\ldots,e_n}= V_e\cap V_{e_1}^c\cap\cdots\cap V_{e_n}^c=\{\xi\in\ftight\mid e\in\xi,e_1\notin\xi,\ldots,e_n\notin\xi\}.\]
In order to simplify the notation, let $E(S)^+=\bigcup_{n=1}^\infty E(S)^n$, where $E(S)^n$ is the Cartesian product of $n$ copies of $E(S)$. For each $n\in\N$ and $\mathbf{e}=(e,e_1,\ldots,e_n)\in E(S)^+$, define $V_{\mathbf{e}}=V_{e:e_1,\ldots,e_n}$. Then the family $\{V_{\mathbf{e}}\}_{\mathbf{e}\in E(S)^+}$ is a basis for the topology on $\ftight$. Notice that if $E(S)$ is countable, which is the case for ultragraphs, then $\ftight$ is second-countable and therefore metrizable.

\begin{remark}\label{rmk:Ve.not.empty}
Given $e\in E(S)\setminus\{0\}$, a simple application of Zorn's lemma shows that there is an ultrafilter containing $\uset{e}$, which in turn implies that $V_e\neq\emptyset$.
\end{remark}

Let $\lspace$ be a weakly-left resolving labelled space. We recall the description of the tight spectrum of $E(S)$, as given in \cite{MR3648984} with the correction pointed out in \cite{Gil3} and \cite{MR4109095}. 
Let $\alpha\in\overline{\awleinf}$  and let $\{\ftg{F}_n\}_{0\leq n\leq|\alpha|}$ (understanding that ${0\leq n\leq|\alpha|}$ means $0\leq n<\infty$ when $\alpha\in\overline{\awinf}$) be a family such that $\ftg{F}_n$ is a filter in $\acfrg{\alpha_{1,n}}$ for every $n>0$ and $\ftg{F}_0$ is either a filter in $\acf$ or $\ftg{F}_0=\emptyset$. Then, the family $\{\ftg{F}_n\}_{0\leq n\leq|\alpha|}$ is  a  \emph{complete family for} $\alpha$ if
\[\ftg{F}_n = \{A\in \acfrg{\alpha_{1,n}} \ | \ r(A,\alpha_{n+1})\in\ftg{F}_{n+1}\}\]
for all $0\leq n < |\alpha|$.

\begin{theorem}\cite[Theorem 4.13]{MR3648984}\label{thm.filters.in.E(S)}
	Let $\lspace$ be a weakly left-resolving labelled space and $S$   its associated inverse semigroup. Then there is a bijective correspondence between filters in $E(S)$ and pairs $(\alpha, \{\ftg{F}_n\}_{0\leq n\leq|\alpha|})$, where $\alpha\in\overline{\awleinf}$ and $\{\ftg{F}_n\}_{0\leq n\leq|\alpha|}$ is a complete family for $\alpha$.
\end{theorem}

Filters are of \emph{finite type} if they are associated with pairs $(\alpha, \{\ftg{F}_n\}_{0\leq n\leq|\alpha|})$ for which $|\alpha|<\infty$, and of \emph{infinite type} otherwise. As a consequence of \cite[Proposition 4.18]{MR3648984}, a filter of finite type is completely determined by $\alpha$ and $\ftg{F}_{|\alpha|}$.

A filter $\xi$ in $E(S)$ with associated labelled path $\alpha\in\overline{\awleinf}$ is sometimes denoted by $\xia$ to stress the word $\alpha$; in addition, the filters in the complete family associated with $\xia$ will be denoted by $\xi^{\alpha}_n$ (or simply $\xi_n$). Specifically,
\begin{align}\label{eq.defines.xi_n}
\xi^{\alpha}_n=\{A\in\acf \ | \ (\alpha_{1,n},A,\alpha_{1,n}) \in \xia\}.
\end{align}

\begin{remark}\label{remark.when.in.xialpha}
	It follows from \cite[Propositions 4.4 and 4.8]{MR3648984} that for a filter $\xi^\alpha$ in $E(S)$ and an element $(\beta,A,\beta)\in E(S)$ we have that $(\beta,A,\beta)\in\xi^{\alpha}$ if and only if $\beta$ is a beginning of $\alpha$ and $A\in\xi^{\alpha}_{|\beta|}$.
\end{remark}

The following theorem characterizes the tight filters in $E(S)$:
\begin{theorem}\cite[Theorems 5.10 and 6.7]{MR3648984} \label{thm:TightFiltersType}
	\label{thm.tight.filters.in.es}
	Let $\lspace$ be a normal labelled space and $S$  its associated inverse semigroup. Then the tight filters in $E(S)$ are:
	\begin{enumerate}[(i)]
		\item The filters of infinite type for which the non-empty elements of their associated complete families are ultrafilters. 
		\item The filters of finite type $\xia$ such that $\xi_{|\alpha|}$ is an ultrafilter in $\acfra$ and for each  $A\in\xi_{|\alpha|}$ at least one of the following conditions hold:
		\begin{enumerate}[(a)]
			\item $\lbf(A\dgraph^1)$ is infinite.
			\item There exists $B\in\acfra$ such that $\emptyset\neq B\subseteq A\cap \dgraph^0_{sink}$.
		\end{enumerate}
	\end{enumerate}
\end{theorem}

For $\alpha\in\overline{\awleinf}$, we denote by $\ftightw{\alpha}$ the set of all tight filters in $E(S)$ for which the associated word is $\alpha$.

\section{Groupoids for ultragraphs via labelled spaces}\label{groupoids}

In this section we associate a groupoid with any ultragraph (in particular, we make no assumptions on the sinks or the range of the edges).  We detail the similarities and differences to the approach in \cite{MR2457327}. For ultragraphs that satisfy Condition~(RFUM2), we show that the unit space of our groupoid (the tight spectrum) coincides with the shift space defined in \cite{tasca2020kms}, and obtain as a consequence that the groupoids also agree. For background on groupoids and their C*-algebras, we refer the reader to \cite{RenaultBook}.

\subsection{The groupoid of a general ultragraph}
\label{vento}

Fix an ultragraph $\G=(G^0,\G^1,r,s)$ and let $\lgraph$ be the labelled graph associated to it as in Definition~\ref{labelfromultra}. Let $\acf$ and $\G^0$ be as in Definition~\ref{prop_vert_gen}.
Observe that if $\alpha=\alpha_1\ldots\alpha_n$ is a path on $\G$, then $\acfra=\acfrg{\alpha_n}$. 

The first step to obtain the groupoid associated to $\G$ is to describe the tight spectrum $\ftight$ of $E(S)$, where $S$ is the inverse semigroup associated to $\lspace$ as in Section \ref{subsection:inverse.semigroup}. For this we need a few auxiliary results, which we prove next. 

\begin{remark}\label{rmk:relative.range}
Notice that for every $A\scj G^0$ and $e\in\G^1$, we have that $r(A,e)=r(e)$ if $s(e)\in A$ and $r(A,e)=\emptyset$ otherwise. In particular $r(\{s(e)\},e)=r(e)$.
\end{remark}

\begin{lema}\label{lem:ultrafilter.equals.principal} 
Let $\alpha$ be a path on $\G$ such that $|\alpha|\geq 1$ and let $\{\ft_n\}_{n=0}^{|\alpha|}$ be a complete family of filters for $\alpha$. If $0\leq n <|\alpha|$, then $\ft_n=\usetr{\{s(\alpha_{n+1})\}}{\acfrg{\alpha_n}}$.
\end{lema}

\begin{proof}
By Remark \ref{rmk:relative.range} and the definition of complete family, $\{s(\alpha_{n+1})\}\in\ft_n$, and if $A\in\ft_n$ then $s(\alpha_{n+1})\in A$. This implies that $\ft_n\scj \usetr{\{s(\alpha_{n+1})\}}{\acfrg{\alpha_n}}$. On the other hand, $\{s(\alpha_{n+1})\}\in\ft_n$ implies that $\usetr{\{s(\alpha_{n+1})\}}{\acfrg{\alpha_n}}\scj \ft_n$.
\end{proof}

\begin{lema}\label{lem:filter.with.finite.set}
Let $\alpha$ be a finite path on $\G$ and $\ft$ be an ultrafilter on $\acfra$. If there exists $A\in\ft$ with $|A|<\infty$, then $\ft=\usetr{\{v\}}{\acfra}$ for some $v\in A$.
\end{lema}

\begin{proof}
Suppose $A=\{v_1,\ldots,v_n\}$. Since $\emptyset\notin\ft$, we have that $n\geq 1$. Since $A=\bigcup_{i=1}^n \{v_i\}$ and $\ft$ is an ultrafilter, by Proposition~\ref{prop:ultrafilter}, there exists $i\in\{1,\ldots,n\}$ such that $\{v_i\}\in\ft$. In this case $\usetr{\{v_i\}}{\acfra}\scj \ft$, and since $\usetr{\{v_i\}}{\acfra}$ is also an ultrafilter, equality holds.
\end{proof}

\begin{lema}\label{lem:infinite.family}
Let $\alpha$ be an infinite path on $\G$, then $\{\ft_n\}_{n=0}^\infty:=\{\usetr{\{s(\alpha_{n+1})\}}{\acfrg{\alpha_n}}\}_{n=0}^\infty$ is the only complete family of filters (in particular, ultrafilters) for $\alpha$.
\end{lema}

\begin{proof}
First suppose that $\{\ft_n\}_{n=0}^\infty$ is a complete family for $\alpha$. Then, it follows from Remark~\ref{rmk:relative.range} and the definition of complete family that $\{s(\alpha_{n+1})\}\in\ft_n$ and that if $A\in\ft_n$, then $s(\alpha_{n+1})\in A$. This implies that $\ft_n\scj \usetr{\{s(\alpha_{n+1})\}}{\acfrg{\alpha_n}}$. On the other hand, $\{s(\alpha_{n+1})\}\in\ft_n$ implies that $\usetr{\{s(\alpha_{n+1})\}}{\acfrg{\alpha_n}}\scj \ft_n$.

That $\{\usetr{\{s(\alpha_{n+1})\}}{\acfrg{\alpha_n}}\}_{n=0}^\infty$ is always a complete family for $\alpha$ follows from Remark~\ref{rmk:relative.range}, and each $\usetr{\{s(\alpha_{n+1})\}}{\acfrg{\alpha_n}}$ is an ultrafilter because it is a principal filter of a singleton.
\end{proof}

\begin{remark}\label{rmk:filters.finite.type}
The above results allow us to describe the set of filters $\filt$ of $E(S)$ for the labelled space associated to an ultragraph. Namely, for each infinite path on $\G$, there is only one filter in $E(S)$ associated to it, and for each finite path $\alpha$, there is a filter in $E(S)$ associated to each filter $\ft$ on $\acfra$. In particular for each $(\alpha,A)\in\mathfrak{p}$, we can define a filter in $E(S)$ by defining $\ft=\usetr{A}{\acfra}$. Since not all filters in $\acfra$ are necessarily principal (see Remark~\ref{rmk:ultraset}), it follows that the correspondence of \cite[Proposition 10]{MR2457327} is incomplete. The issue lies in Case I of the proof of \cite[Proposition 10]{MR2457327}, because, contrary to what happens with graphs, if the product of two idempotents is not zero, then the idempotents do not necessarily correspond to two ultrapaths where one is the beginning of the other. See also \cite[Proposition 4.4]{MR3648984}.
\end{remark}

We now give a complete description of $\ftight$.

\begin{proposicao}\label{prop:tight.filters}
For each infinite path $\alpha$ on $\G$, there is unique element $\xi\in\ftight$, whose associated word is $\alpha$. And if $\xia$ is a filter of finite type, then $\xia\in\ftight$ if and only if one of the following holds:
\begin{enumerate}[(i)]
    \item There exists $v\in G^0_s$ such that $\xi_{|\alpha|}=\usetr{\{v\}}{\acfra}$.
    \item For all $A\in\xi_{|\alpha|}$, $|\varepsilon(A)|=\infty$.
    \item For all $A\in\xi_{|\alpha|}$, $|A\cap G^0_s|=\infty$.
\end{enumerate}
\end{proposicao}

\begin{proof}
First, if $\alpha$ is an infinite path on $\G$, then, by Lemma~\ref{lem:infinite.family}, the family $\{\usetr{\{s(\alpha_{n+1})\}}{\acfrg{\alpha_n}}\}_{n=0}^\infty$ is complete and each filter is an ultrafilter. By Theorem~\ref{thm.tight.filters.in.es}, the filter associated to $\alpha$ and this family is an element of $\ftight$. The uniqueness of such $\xi$ follows from Lemma~\ref{lem:infinite.family}.

For the second part, suppose that $\xia$ is a filter of finite type such that $\xia\in\ftight$. By Lemma \ref{lem:filter.with.finite.set}, if there exists $A\in\xi_{|\alpha|}$, with $|A|<\infty$, then $\xi_{|\alpha|}=\usetr{\{v\}}{\acfra}$ for some $v\in r(\alpha)$. Since $\xi$ is tight, by Theorem~\ref{thm.tight.filters.in.es}, either $v\in G^0_s$ and (i) is satisfied, or $v$ is an infinite emitter and (ii) is satisfied. Suppose now that $|A|=\infty$ for all $A\in \xi_{|\alpha|}$. In this case, either (ii) is satisfied, or there exists $B\in \xi_{|\alpha|}$ such that $|\varepsilon(B)|<\infty$. In the latter case, given $A\in\xi_{|\alpha|}$, we have that $|\varepsilon(A\cap B)|\leq|\varepsilon(B)|<\infty$. Since $\xi_{|\alpha|}$ is an ultrafilter, and thus $ A\cap B \in \xi_{|\alpha|}$, it follows that $|A\cap B|=\infty$. Hence $A\cap B$ has an infinite amount of sinks, and the same holds for $A$.

Finally if $\xia$ is a filter of finite type such that $\xi_{|\alpha|}$ satisfies Conditions~(i) or (iii) in the statement, then $\xi_{|\alpha|}$ satisfies Condition~(ii)(b) in Theorem~\ref{thm.tight.filters.in.es}, and if $\xi_{|\alpha|}$ satisfies Condition~(ii) in the statement, then $\xi_{|\alpha|}$ satisfies Condition~(ii)(a) in Theorem~\ref{thm.tight.filters.in.es}. In all cases $\xia\in\ftight$.
\end{proof}

Using the description of filters in $E(S)$ given in Remark~\ref{rmk:filters.finite.type} and Proposition~\ref{prop:tight.filters}, the elements of $\ftight$ can described in two ways:

\begin{itemize}
\item if $\xi\in\ftight$ is of infinite type, then it is completely described by the infinite path associated to it;
\item if $\xi\in\ftight$ is of finite type, then it is described by a pair $(\alpha,\ft)$, where $\alpha$ is the labelled path associated to $\xi$ and $\ft$ is a filter in $\acfra$ satisfying one of the three conditions of Proposition~\ref{prop:tight.filters}. The filter $\ft$ is equal to $\xi_{|\alpha|}$ and by the definition of complete family, for all $n<|\alpha|$, $\xi_n$ can be recuperated from $\ft$.
\end{itemize}

Next, associate a topological groupoid to an arbitrary ultragraph as the groupoid associated to its corresponding labelled space (see  \cite{Gil3} for groupoids associated to general labelled spaces). This can be done in terms of a shift map and the Renault-Deaconu construction \cite{Deaconu95,Renault00}. For each $n\in\N$, we let $\ftight^{(n)}=\{\xia\in\ftight\mid |\alpha|\geq n\}$. The map $\sigma:\ftight^{(1)}\to\ftight$ of \cite[Proposition 4.8]{Gil3} can be described in our case as:
\begin{enumerate}
    \item If the associated path of $\xi$ is an infinite path $\alpha_1\alpha_2\alpha_3\ldots$, then $\sigma(\xi)$ is the filter associated to the path $\alpha_2\alpha_3\ldots$ as in Proposition \ref{prop:tight.filters}.
    
    \item If the associated path of $\xi$ is a finite path $\alpha_1\alpha_2\ldots\alpha_n$ with $n\geq 2$, then $\sigma(\xi)$ is the filter associated to the pair $(\alpha_2\ldots\alpha_n,\xi_n)$.
    
    \item If the associated path of $\xi$ is an edge $e\in\G^1$, then $\sigma(\xi)$ is the filter associated to the pair $(\eword,\usetr{\xi_1}{\acf})$, recalling that $\eword$ is the empty word.
\end{enumerate}

For $n\in\N^*$, in order to apply $\sigma^n$ to a element $\xi\in\ftight$, it is necessary that $\xi\in\ftight^{(n)}$. The Renault-Deaconu groupoid associated to the $\sigma$ is then
\[\Gamma(\ftight,\sigma)=\{(\xi,m-n,\eta)\in\ftight\times\mathbb{Z}\times\ftight\mid m,n,\in\N,\xi\in\ftight^{(n)},\eta\in\ftight^{(m)},\sigma^n(\xi)=\sigma^m(\eta)\},\]
with product and inverse given by
\begin{itemize}
    \item $(\xi,k,\eta)(\zeta,l,\rho)=(\xi,k+l,\rho)$ for $(\xi,k,\eta),(\zeta,l,\rho)\in\Gamma(\ftight,\sigma)$ such that $\eta=\zeta$, and 
    \item $(\xi,k,\eta)^{-1}=(\eta,-k,\xi)$ for $(\xi,k,\eta)\in\Gamma(\ftight,\sigma)$.
\end{itemize}
respectively. A basis for the topology on $\Gamma(\ftight,\sigma)$ is given by the family of sets of the form
\[\mathcal{V}(U,V,m,n)=\{(\xi,k,\eta)\in\Gamma(\ftight,\sigma)\mid k=m-n, (\xi,\eta)\in U\times V,\sigma^m(\xi)=\sigma^n(\eta)\},\]
where $m,n\in\N$, $U$ is an open subset of $\ftight^{(m)}$ and $V$ is an open subset of $\ftight^{(n)}$.

\begin{remark}\label{rmk:ultraset}
Due to Remark \ref{rmk:filters.finite.type}, there is not always a bijection between the groupoid $\Gamma(\ftight,\sigma)$ and the ultragraph groupoid of \cite{MR2457327} for ultragraphs without sinks, however both approaches are very similar. The idea in both is to consider an inverse semigroup associated to an ultragraph, then consider the space of filters in the idempotent semilattice and then take a certain closed subspace. In general, the closed subspace is the tight spectrum defined by Exel in \cite{MR2419901}, which appeared around the same time as  \cite{MR2457327}. If an ultragraph has no sinks, then because of \cite[Theorem 5.10]{MR3648984} the ultrafilters in $E(S)$ are exactly the filters associated with the infinite paths. In this case, the set $\ftight$ is the closure of the ``set of infinite paths'' as it is done in the discussion after \cite[Proposition 18]{MR2457327}. In order to describe which points of $\ftight$ are missing in the construction of \cite{MR2457327}, we recall that the approach there is to define an ultraset as an element $A\in\G^0$ such that the principal filter $\usetr{A}{\G^0}$ is an ultrafilter. Because all singleton subsets of $G^0$ are in $\G^0$, we see that a principal filter is an ultrafilter if and only if $A$ is a singleton. This means that  the approach in \cite{MR2457327} only considers elements of finite type in $\ftight$ that are given by pairs $(\alpha,\usetr{\{v\}}{\acfra})$, where $v\in G^0$ is an infinite emitter. In fact, their groupoid will coincide with $\Gamma(\ftight,\sigma)$ if and only if the range of every edge is finite. To see this, note that if every range is finite, then every element of $\acf$ is a finite subset of $G^0$, and therefore, by Lemma \ref{lem:filter.with.finite.set}, every ultrafilter is principal. On the other hand, if there is an edge $e$ such that $r(e)$ is infinite, then by Zorn's lemma, there exists an ultrafilter $\ft$ containing all cofinite subsets of $r(e)$. Notice that, since we are assuming that the ultragraph has no sinks, all elements of $\ft$ satisfies (ii) of Proposition \ref{prop:tight.filters} and $\ft$ does not correspond to any vertex, and therefore the filter given by the pair $(e,\ft)$ has no correspondence in the groupoid defined in \cite{MR2457327}. See also Proposition~\ref{prop:correspondence.ultrafilters} and  Example~\ref{picapau}.
\end{remark}

\subsection{The groupoid of an ultragraph that satisfies Condition~(RFUM2)}

The tight spectrum (and hence the groupoid) associated with a general ultragraph is quite technical, which makes it hard to use in some applications. For ultragraphs that satisfy Condition~(RFUM2), we can avoid using filters altogether, and build a groupoid which is more conducive to applications, see \cite{GDD, tasca2020kms} (cf. \cite{BCW} where the description of the groupoid in the graph case is given). The unit space of this groupoid (also known as the boundary path space), with the associated shift map, is proposed to be a version of a shift of finite type over an infinite alphabet in \cite{MR3938320, tasca2020kms}, and the dynamics of these shifts is studied in \cite{SoboG, gonccalves2018li, gonccalves2019ultragraph}, among other papers. In this subsection, we show that the unit space of the groupoid defined in \cite{tasca2020kms} agrees with the tight spectrum $\ftight$ associated with the ultragraph. Consequently, the groupoid introduced in \cite{GDD, tasca2020kms} agrees with the groupoid we described in Subsection~\ref{vento}.

\subsubsection{The boundary path space of an ultragraph that satisfies Condition~(RFUM2)} \label{sec:boundary.path.RFUM2}

In this subsection we recall the definition of the boundary path space of an ultragraph that satisfies Condition~(RFUM2), as defined in \cite{tasca2020kms}. We start with the definition of minimal infinite emitters.

\begin{definicao} \label{def_G^0}
	Let $\G$ be an ultragraph. For each $A\in \G^0$, we define $\varepsilon(A):=\{e\in \G^1:s(e)\in A\}$. We say that $A\in \G^0$ is an infinite emitter if $|\varepsilon(A)|=\infty$. Otherwise we say that $A$ is a finite emitter. 
 We say that $A$ is a \emph{minimal infinite emitter} if $A$ is an infinite emitter, contains no proper subsets (in $\G^0$) that are infinite emitters, and contains no proper subsets (in $\G^0$) which are finite emitters and have
	infinite cardinality. We denote by $A_\infty$ the set of all minimal infinite emitters in $\G^0$.
\end{definicao}

\begin{remark} Notice that the notion of minimal infinite emitters is meaningless if we allow relative complements in $\G^0$, but this is not the case, as we pointed out in Remark~\ref{rmk:no.relative.complement}. Nevertheless this point highlights the subtleties involved in the definition of $\G^0$.
\end{remark}

The set of sinks in an ultragraph is denoted by $G^0_s$. Let $\G^0_s:=\displaystyle\bigsqcup_{v_i\in G^0_s} \{\{v_i\}\}\subseteq \G^0$, that is, $\G^0_s$ is the collection of all singletons of $\G^0$ whose element is a sink. We now define the sets that are minimal sinks.

\begin{definicao}\label{tortuga}
	Let $\G$ be an ultragraph and $A\in \G^0$. We say that $A$ is a \emph{minimal sink} if $|A|=\infty$, $|\varepsilon(A)|<\infty$ and $A$ has no subsets (in $\G^0$) with infinite cardinality. We denote by $A_s$ the set of all minimal sinks in $\G^0$. 
\end{definicao}

Define the \textit{boundary path space} as the set \[X:=\p^\infty \sqcup X_{min}\sqcup X_{sin},\]
where 
\[X_{min}:=\{ (\alpha,A)\in\p:|\alpha|\geq 1, A\in A_\infty\}\cup
	\{ (A,A)\in \p^0:A\in A_\infty\},\]
	\[X_{sin}:=\{ (\alpha,A)\in\p:|\alpha|\geq 1, A\in \G^0_s\sqcup A_s\}\cup \{ (A,A)\in \p^0:A\in \G_s^0\sqcup A_s\}.\]
Let $X_{fin}:=X_{min}\sqcup X_{sin}$. We define a topology on $X$ as follows.
For each $(\beta,B)\in \p$, let  
$$D_{(\beta,B)}:=\left\{y\in X: y=\beta\gamma ; s(\gamma)\in B\right\},$$
and for each $(\beta,B)\in X_{fin}$, $F\subseteq \varepsilon(B)$ finite, and $S\subseteq B\cap G_{s}^0$ finite, define 
\begin{align*}
D_{(\beta,B),F,S}:=\left\{(\beta,B)\right\}\sqcup&\left\{ y\in X:y=\beta\gamma',|\gamma'|\geq 1, \gamma_1'\in\varepsilon (B)\setminus F \right\}\\ 
\sqcup &\left\{y\in X_{sin}:y=(\beta,\{v\}):v\in B\backslash S\right\}.
\end{align*}	
Then the collection of cylinders \[\big\{\{ D_{(\beta,B)}:(\beta,B)\in\p, |\beta|\geq 1\}\cup\{D_{(\beta,B),F,S}:(\beta,B)\in X_{fin},F\subseteq\varepsilon(B),S\subseteq B,|F|,|S|<\infty\}\big\}\] is a countable basis for a topology on $X$. 
If $\G$ satisfies Condition~(RFUM2) (defined below), then these cylinders are compact, \cite[Proposition~3.22]{tasca2020kms}.

We recall the definition of Condition~(RFUM2):

{\textbf{Condition (RFUM2):}}\label{RFUM2} We say that an ultragraph $\G$ satisfies Condition~(RFUM2) if for each edge $e\in \G^1$ its range can be written as $$r(e)=\bigcup_{n=1}^k A_n,$$
where $A_n$ is either a minimal infinite emitter, or a minimal sink, or a singleton formed by a sink or a regular vertex.

\begin{remark}
For ultragraphs with no sinks Condition~(RFUM2) as defined above coincides with Condition~(RFUM) of \cite{MR3938320}.
\end{remark}

\subsubsection{The tight spectrum for an ultragraph that satisfies Condition~(RFUM2)}

In this subsection we  describe the tight spectrum of an ultragraph that satisfy Condition~(RFUM2) and identify it with the space $X$ of the previous subsection. Before we prove this we need a sequence of auxiliary results.

Recall that $\acf$ is the smallest family of subsets of $G^0$ that contains $\G^0$ and is closed under relative complements, finite unions and finite intersections (see Definition~\ref{prop_vert_gen}). We have the following description of $\acf$.

\begin{lema}\label{lem:description.B}
Let $\acf$ be as above. Then
\begin{equation}\label{eq:description.B}
\acf=\left\{\bigcup_{i=1}^n A_i\setminus B_i \mid n\in\N, \text{ and } A_i,B_i\in\G^0\ \forall i\in\{1,\ldots,n\}\right\}.
\end{equation}
Moreover, for any finite path $\alpha$ on $\G$,
\[\acfra=\left\{\bigcup_{i=1}^n A_i\setminus B_i \mid n\in\N, \text{ and } A_i,B_i\in\G^0,\ A_i,B_i\scj r(\alpha) \ \forall i\in\{1,\ldots,n\}\right\}.\]
\end{lema}

\begin{proof}
Let $\acfg{C}$ be the set defined in the right hand side of Equation~\ref{eq:description.B}. Clearly $\acfg{C}\scj \acf$, and since $\emptyset\in\G^0$, we have that $\G^0\scj\acfg{C}$. Using that $\G^0$ is closed under finite intersections and finite unions, as well as some basic properties of operations on sets, we see that $\acfg{C}$ is closed under relative complements, finite unions and finite intersections. By the minimality of $\acf$, we conclude that $\acfg{C}=\acf$. The last part is similar.
\end{proof}

Recall that $A_s$ and $A_{\infty}$ denote the minimal sinks and the minimal infinite emitters, respectively (see Definition~\ref{def_G^0} and Definition~\ref{tortuga}).

\begin{lema}\label{lem:intersection.minimal}
Suppose that $A,B\in\G^0$.
\begin{enumerate}[(i)]
    \item If $A,B\in A_{\infty}$, then $A=B$ or $|A\cap B|<\infty$.
    \item If $A,B\in A_s$, then $A=B$ or $|A\cap B|<\infty$.
    \item If $A\in A_{\infty}$ and $B\in A_s$, then $|A\cap B|<\infty$.
\end{enumerate}
\end{lema}

\begin{proof}
(i) This statement follows from \cite[Proposition 3.6]{tasca2020kms}.

(ii) This is the statement of \cite[Proposition 3.9]{tasca2020kms}.

(iii) Notice that $A\cap B\in\G^0$ and $\varepsilon(A\cap B)\leq\varepsilon(B)<\infty$. By the definition of $A_{\infty}$, $|A\cap B|<\infty$.
\end{proof}

In the next couple of lemmas we characterize the ultrafilters in $\acfrg{\alpha}$.

\begin{lema}\label{lem:ultrafilters.rfum.1}
Suppose that $\G$ is an ultragraph that satisfies condition (RFUM2). Let $\alpha\in\awstar$ and $\ft$ be an ultrafilter on $\acfrg{\alpha}$.
\begin{enumerate}[(i)]
    \item If there exists $A\in\ft$ with $|A|<\infty$, then $\ft=\usetr{\{v\}}{\acfrg{\alpha}}$ for some unique $v\in r(\alpha)$.
    \item If for every $A\in\ft$ it holds that $|A|=\infty$, then there exists a unique $B\in A_s\sqcup A_{\infty}$ such that $B\in \ft$.
    
\end{enumerate}
\end{lema}

\begin{proof}
(i) This follows from Lemma \ref{lem:filter.with.finite.set}. The uniqueness of $v$ follows from the fact that $\{v\}\cap\{w\}=\emptyset$ if $v\neq w$ and $\emptyset\notin\ft$.

(ii) Since $\ft$ does not contain a finite subset of $G^0$, by Lemma~\ref{lem:description.B} and the fact it is a filter, $\ft$ contains an element of $\G^0$. By \cite[Lemma 7.12]{tasca2020kms} and since $\ft$ is a prime filter, there exists $B\in A_s\sqcup A_\infty$ such that $B\in\ft$. The uniqueness of such $B$ follows from Lemma \ref{lem:intersection.minimal} and the fact that $\ft$ is closed under intersections.
\end{proof}

\begin{lema}\label{lem:ultrafilters.rfum.2}
Suppose that $\G$ is an ultragraph that satisfies condition (RFUM2) and let $\alpha\in\awstar$. If $B\in A_s\cup A_{\infty}$ is such that $|B|=\infty$ and $B\scj r(\alpha)$, then there exists a unique ultrafilter $\ft$ in $\acfra$ such that $B\in\ft$ and $|A|=\infty$ for all $A\in\ft$.
\end{lema}

\begin{proof}
For the existence, we let 
\[\ft=\{C\in \acfrg{\alpha}\mid \text{there exists } A\in\acfrg{\alpha}\text{ with }|A|<\infty\text{ such that } C\supseteq B\setminus A\}.\]
It is easy to check that $\ft$ is a filter such that $B\in\ft$ and $|C|=\infty$ for all $C\in\ft$. By Proposition \ref{prop:ultrafilter}, in order to prove that $\ft$ is an ultrafilter, we will show that if $D\in\acfra$ is such that $D\cap C\neq\emptyset$ for every $C\in\ft$, then $D\in\ft$. Using Lemma \ref{lem:description.B} and \cite[Lemma 7.12]{tasca2020kms}, we can write $D=\bigcup_{i=1}^n X_i\setminus Y_i$ where $X_i$ is either finite or $X_i\in A_s\cup A_{\infty}$ and $X_i\neq X_j$ if $i\neq j$. By Lemma \ref{lem:intersection.minimal}, there exists $i_0$ such $X_{i_0}=B$, otherwise $D\cap (B\setminus A)=\emptyset$, where $A=\bigcup_{i=1}^n(B\cap X_i)$ is finite set. We may assume that $i_0=1$. Also, by \cite[Lemma 7.12]{tasca2020kms}, $Y_1=\bigcup_{j=1}^m A_j$, where $A_j$ is either finite or $A_j\in A_s\cup A_{\infty}$. Notice that $A_j\neq B$ for all $j$, otherwise $X_1\setminus Y_1=\emptyset$, and arguing as above we would find $A\scj G^0$ finite such that $D\cap (B\setminus A)=\emptyset$. Let $Z_1=\bigcup_{j=1}^m(B\cap A_i)$ and $Z_2=\bigcup_{i=2}^n(B\cap X_i)$ and notice they are both finite sets. It follows that 
\[D\supseteq D\cap (B\setminus Z_2)= (B\setminus Y_1)\cap (B\setminus Z_2)=(B\setminus Z_1)\cap (B\setminus Z_2)=B\setminus(Z_1\cup Z_2),\]
which implies that $D\in\ft$.

For the uniqueness, we let $\ftg{H}$ be another ultrafilter in $\acfra$ such that $B\in\ftg{H}$ and $|A|=\infty$ for all $A\in \ftg{H}$. We prove that $\ft\scj\ftg{H}$. For that it is sufficient to prove that $B\setminus A\in\ftg{H}$ for all $A\in\acfra$ finite. Fixed such $A$, notice that $B\scj B\cup A=(B\setminus A)\cup A$. Since $\ftg{H}$ is prime and $A$ is finite (so $A\notin \ftg{H})$, we have that $B\setminus A\in\ftg{H}$. Since $\ft$ is also an ultrafilter, we conclude that $\ft=\ftg{H}$.
\end{proof}

Recall that $\widehat{\acfra}$ denotes the Stone dual of $\acfra$ (see Subsection~\ref{subsection:filters.and.characters}). Joining the previous two lemmas we get the following result.

\begin{proposicao}\label{prop:correspondence.ultrafilters}
Suppose that $\G$ is an ultragraph that satisfies condition (RFUM2) and let $\alpha\in\awstar$. There is a bijective correspondence between $\{\{v\}\mid v\in r(\alpha)\}\cup\{B\in A_s\cup A_{\infty}\mid B\scj r(\alpha),\ |B|=\infty\}$ and $\widehat{\acfra}$ that maps $\{v\}$, where $v \in r(\alpha)$, to $\usetr{\{v\}}{\acfrg{\alpha}}$, and maps $B \in A_s\cup A_{\infty}$ such that $B\scj r(\alpha)$ and $|B|=\infty$ to $\{C\in \acfrg{\alpha}\mid \text{there exists } A\in\acfrg{\alpha}\text{ with }|A|<\infty\text{ such that } C\supseteq B\setminus A\}$.
\end{proposicao}

\begin{proof}
The proof follows from Lemmas \ref{lem:ultrafilters.rfum.1} and \ref{lem:ultrafilters.rfum.2}.
\end{proof}

We now prove the main result of this subsection.

\begin{teorema}\label{thm:ftight.homeo.X}
Suppose that $\G$ is an ultragraph that satisfies condition (RFUM2). For each $\alpha\in\awstar$, let $\phi_{\alpha}:\widehat{\acfra}\to \{\{v\}\mid v\in r(\alpha)\}\cup \{B\in A_s\cup A_{\infty}\mid B\scj r(\alpha)\}$ be the correspondence given in Proposition \ref{prop:correspondence.ultrafilters}. Then, the map $\Phi:\ftight\to X$ given by
\[\Phi(\xia)=\begin{cases}
(\phi_{\eword}(\xi_0),\phi_{\eword}(\xi_0)), & \text{if }\alpha=\eword; \\
(\alpha,\phi_{\alpha}(\xi_{|\alpha|})), & \text{if }0<|\alpha|<\infty; \\
\alpha, & \text{if }|\alpha|=\infty
\end{cases}\]
is a homeomorphism.
\end{teorema}
   
\begin{proof}
We start by showing that $\Phi$ is well-defined. If $\xia$ is a filter of infinite type, then it is clear that $\Phi(\xia)\in X$. If $\xia$ is a filter of finite type and $\phi_{\alpha}(\xi_{|\alpha|})\in A_s\cup A_{\infty}$, then $\Phi(\xia)$ is also an element of $X.$ Finally, we have to show that if $\xia$ is a filter of finite type and $\phi_{\alpha}(\xi_{|\alpha|})$ is a singleton, say $\{v\}$, then $v$ is either a sink or an infinite emitter. This follows from Theorem \ref{thm:TightFiltersType}, since in this case $\{v\}\in\xi_{|\alpha|}$. 

We now build an inverse for $\Phi$. Denote by  $G^0_{sg}$ the set of vertices that are either an infinite emitter or a sink. Define $\Psi:X\to \ftight$ as follows:

\begin{enumerate}[(a)]
    \item For $(A,A)$, where $A\in\{\{v\}\mid v\in G^0_{sg}\}\cup A_s\cup A_{\infty}$, we let $\Psi(A,A)$ be the filter associated to the pair $(\eword,\phi_{\eword}^{-1}(A))$.
    \item For $(\alpha,A)$, where $A\in\{\{v\}\mid v\in r(\alpha)\cap G^0_{sg}\}\cup \{B\in A_s\cup A_{\infty}\mid B\scj r(\alpha)\}$, we let $\Psi(\alpha,A)$ be the filter associated to the pair $(\alpha,\phi_{\alpha}^{-1}(A))$.
    \item For $\alpha$ an infinite path, we let $\Psi(\alpha)$ be the filter associated to $\alpha$.
\end{enumerate}

We  check that $\Psi$ is well-defined. For (a) and (b), fix $\alpha\in\awstar$, where $\alpha=\eword$ for item $(a)$. We have four cases to consider for $A$. If $A=\{v\}$ with $v$ a sink, then all elements of $\phi_{\alpha}^{-1}(A)$ contains $\{v\}$ as a subset and so they satisfy Theorem \ref{thm:TightFiltersType}(ii)(b). If $A=\{v\}$ with $v$ an infinite emitter, then all elements of $\phi_{\alpha}^{-1}(A))$ are infinite emitters and so they satisfy Theorem \ref{thm:TightFiltersType}(ii)(a). If $A\in A_s$, then it contains an infinite number of sinks and hence, by the description of $\phi_{\alpha}^{-1}(A)$, we see that all its elements satisfy Theorem \ref{thm:TightFiltersType}(ii)(b). Analogously, if $A\in A_{\infty}$ and $|A|=\infty$, all elements of $\phi_{\alpha}^{-1}(A))$ satisfy Theorem \ref{thm:TightFiltersType}(ii)(a). In all cases, we get a tight filter. For (c), $\Psi(\alpha)$ is well-defined by Proposition \ref{prop:tight.filters} and Theorem \ref{thm:TightFiltersType}(i).

To see that $\Psi=\Phi^{-1}$, for finite paths, we use \cite[Propositions~4.3~and~4.4]{MR3648984} and Proposition~\ref{prop:correspondence.ultrafilters}, and, for infinite paths, we use \cite[Theorem~4.13]{MR3648984} and Proposition \ref{prop:tight.filters}.

Let us prove now that $\Phi$ is continuous. For this we consider the basis for  the topology on $X$ given in  in Section~\ref{sec:boundary.path.RFUM2}, and we show that $\Phi^{-1}$ of each such basic open set is a basic open set for $\ftight$. If $(\beta,B)\in\up$, then, by the description of $\Phi^{-1}$ as $\Psi$, Remark \ref{rmk:relative.range} and \cite[Proposition~4.1]{MR3648984}, we see that $\Phi^{-1}(D_{(\beta,B)})=V_{(\beta,B,\beta)}$. Analogously, if $(\beta,B)\in X_{fin}$, $F=\{f_1,\ldots,f_n\}\scj\varepsilon(B)$ and $S=\{s_1,\ldots,s_m\}\scj B\cap G_s^0$, then \[\Phi^{-1}(D_{(\beta,B),F,S})=V_{(\beta,B,\beta):(\beta f_1,r(f_1),\beta f_1),\ldots,(\beta f_n,r(f_n),\beta f_n),(\beta,\{s_1\},\beta),\ldots,(\beta,\{s_m\},\beta)}.\]
Hence the inverse image of each basic open set of $X$ by $\Phi$ is also open, from where we conclude that $\Phi$ is continuous.

It remains to show that $\Psi$ is continuous. For that, it is sufficient to show that $\Psi$ is a local homeomorphism. Notice that the sets $D_{(\beta,B)}$, for $(\beta,B)\in\mathfrak{p}$, cover $X$, and that, for each $(\beta,B)\in\mathfrak{p}$, we can restrict $\Phi$ to a bijection between $V_{(\beta,B,\beta)}$ and $D_{(\beta,B)}$. Since $V_{(\beta,B,\beta)}$ is compact, $X$ is Hausdorff and $\Phi$ is continuous, this restriction is a homeomorphism with inverse given by a restriction of $\Psi$. This implies that $\Psi$ is a local homeomorphism and therefore continuous.

\end{proof}

\begin{exemplo}\label{picapau}
Let $\G$ be the ultragraph such that $G^0=\{v_i\}_{i\in \N}$ and edges $e_i$ such that $s(e_0)=v_0$, $r(e_0)= \{v_i: i = 1, 2, \ldots \}$ and, for $i\neq 0$, $s(e_i)=v_i$ and $r(e_i)=\{ v_0, v_i\}$. Notice that this ultragraph has no sinks and satisfies condition (RFUM2). The set $r(e_0)$ is a minimal infinite emitter so that $(e_0,r(e_0))\in X$. In order to describe the corresponding element of $\ftight$, we note that $\acfrg{e_0}$ consists only of the finite and cofinite subsets of $r(e_0)$. Then, by Proposition~\ref{prop:correspondence.ultrafilters} and what we have just noticed, $\phi_{e_0}(r(e_0))$ is the set of all cofinite subsets of $r(e_0)$, which is not a principal filter. Using the homeomorphism from Theorem~\ref{thm:ftight.homeo.X}, $(e_0,r(e_0))$ corresponds to the element $\xi$ given by the pair $(e_0,\phi_{e_0}^{-1}(r(e_0))$. Finally, notice that, by the  description of convergence in $X$ given in \cite[Corollary 3.20]{tasca2020kms}, the sequence of infinite paths $(e_0e_n^{\infty})_{n\in\N^*}$ converges to $(e_0,r(e_0))$. Notice that in the space defined just after \cite[Proposition~18]{MR2457327}, there is no element corresponding to $\xi$ (see Remark~\ref{rmk:ultraset}).
\end{exemplo}

\subsubsection{Identifying the groupoids}

In \cite{tasca2020kms} the groupoid associated with an ultragraph that satisfies Condition~(RFUM2) is the Renault-Deaconu groupoid obtained from a partially defined shift map on the space $X$, in a similar fashion to what we described for general ultragraphs in Subsection~\ref{vento}. In this subsection we identify our general groupoid described in Subsection~\ref{vento} with the groupoid described in \cite{tasca2020kms}.
For completeness sake we briefly recall the definition of the groupoid defined in $\cite{tasca2020kms}$.

	\begin{definicao}\label{grupoide}
		Let $\G$ be an ultragraph that satisfies Condition~(RFUM2). We define the \emph{groupoid associated to $\G$} by $G(X,\tilde{\sigma}):=\{(x,m-n,y):x,y\in X;m,n\in \N;\tilde{\sigma}^m(x)=\tilde{\sigma}^n(y) \}$, where $\tilde{\sigma}$ is the shift map on $X$ (see \cite{tasca2020kms}). The source function is defined by $s(x,k,y)=y$ and the range by $r(x,k,y)=x$. Multiplication is given by $(x,k,y)(y,l,z)=(x,k+l,z)$ and inversion by $(x,k,y)^{-1}=(y,-k,x)$. The topology considered is that of cylinder sets, see \cite{tasca2020kms}.
	\end{definicao}

\begin{proposicao}
Suppose that $\G$ is an ultragraph that satisfies condition (RFUM2). Then the groupoids $G(X,\tilde{\sigma})$ and $\Gamma(\ftight,\sigma)$ are isomorphic as topological groupoids.
\end{proposicao}

\begin{proof}
It easy to check that $\Phi$ from Theorem \ref{thm:ftight.homeo.X} gives a conjugation between the shift on $X$ and the shift on $\ftight$, and since both groupoids are Renault-Deaconu groupoids, the result follows.
\end{proof}




\section{Partial actions for ultragraphs via labelled spaces}

Partial actions are the subject of intense research. They are interesting on their own and their associated algebras are used to model many known algebras, in particular the ones associated with combinatorial objects (see \cite{MR2419901, MR3699795}). In the context of ultragraphs, partial actions are used to  describe  the  C*-algebras  associated  with ultragraphs that satisfy Condition~(RFUM) or Conditon~(RFUM2), see \cite{MR3938320, tasca2020kms}. In the purely algebraic setting, the ultragraph Leavitt path algebra associated with a general ultragraph is realized as a partial skew group ring associated to a partial action in \cite{goncalves_royer_2019}. The goal of this section is to show that the topological partial action  associated with a labelled space (as in \cite[Proposition 3.12]{MR4109095}) gives rise to an algebraic partial action (via locally constant functions with compact support) that is equivalent to the algebraic partial action defined in \cite{goncalves_royer_2019}. Furthermore, using the identification of $\ftight$ with $X$ from Theorem~\ref{thm:ftight.homeo.X}, we also observe that for ultragraphs that satisfy Condition~(RFUM2) the partial action defined in \cite{MR4109095} generalizes that of \cite{tasca2020kms}.

\subsection{The topological partial action on the tight filters of $E(S)$} \label{s:partialaction}
In \cite{MR4109095} a partial action by the free group generated by the alphabet $\alf$ is defined  on the tight spectrum $\ftight$ of a labelled space. In this section we outline this construction and refer the reader to \cite[Section 3]{MR4109095} for the details. For the definition of a topological partial action we refer the reader to \cite[Proposition 2.5 and Definition 5.1]{MR3699795}. 

Fix a weakly-left resolving labelled space $\lspace$ and let $\mathbb{F}$ be the free group generated by $\alfg{A}$ (identifying the identity of $\mathbb{F}$ with $\eword$). Then, by \cite[Proposition 3.12]{MR4109095}, for every $t\in \mathbb{F}$ there is a compact open set $V_t\subseteq \ftight$ and a homeomorphism $\varphi_t:V_{t^{-1}}\to V_t$ such that 
\begin{equation} \label{eqn:labelled.top.partial.action}
\Delta=(\{V_t\}_{t\in\F},\{\varphi_t\}_{t\in\F})    
\end{equation}
is a topological partial action of $\F$ on $\ftight$. In particular, $V_\eword=\ftight$ and if 
$\alpha,\beta\in \awstar$, then $V_\alpha=V_{(\alpha,r(\alpha),\alpha)}, V_{\alpha^{-1}}=V_{(\eword,r(\alpha),\eword)}$, and $V_{(\alpha\beta^{-1})^{-1}}= \varphi_{\beta^{-1}}^{-1}(V_{\alpha^{-1}})$, with $V_{(\alpha\beta^{-1})^{-1}}\neq\emptyset$ if and only if $r(\alpha)\cap(\beta)\neq\emptyset$ (\cite[Lemma 3.10]{MR4109095}). 

For the labelled space associated to an ultragraph, we can intuitively describe the map $\varphi_{\alpha\beta^{-1}}$, for $\alpha,\beta\in\awstar$, as cutting $\beta$ from the beginning of an element $\xi\in V_{\alpha^{-1}\beta}$ and then gluing $\alpha$ in front of it. For filters of finite type we just have to take into account the last filter of the corresponding complete family. In most cases, the last filter is kept fixed, unless the empty word is involved. If $\beta$ is the labelled path associated to an element $\xi\in V_{\alpha^{-1}\beta}$, then by cutting $\beta$, we get the filter associated to the pair $(\eword,\usetr{\xi_{|\beta|}}{\acf})$, and by gluing $\alpha$ afterwards, we get the filter associated to the pair $(\alpha,\ft)$, where $\ft=\{A\cap r(\alpha)\mid A\in \usetr{\xi_{|\beta|}}{\acf}\}$.

\begin{remark}
Using the map $\Phi$ of Theorem~\ref{thm:ftight.homeo.X} and the above description of the partial action, it follows that $\Delta$ is a generalization of the partial action for ultragraphs without sinks satisfying (RFUM) as in \cite{MR3938320} and the generalization to ultragraphs satisfying (RFUM2) in \cite{tasca2020kms}.
\end{remark}

We let 
\begin{equation} \label{dfn:trans.groupoid}
    \mathbb{F} \ltimes_{\Delta} \ftight=\{(t,\xi)\in \mathbb{F} \times \ftight: \xi\in V_{t}\}
\end{equation}
denote the transformation groupoid associated with $\Delta$, as defined in \cite{MR2045419}. Then $\Gamma(\ftight,\sigma)$ defined in Section \ref{groupoids} and $\mathbb{F} \ltimes_{\Delta} \ftight$ are isomorphic as topological groupoids, \cite[Theorem 5.5]{MR4109095}. In Theorem~\ref{kart} and Section~\ref{sec:cores} it will be convenient to use $\mathbb{F} \ltimes_{\Delta} \ftight$ as the groupoid associated with a labelled space.

\subsection{The algebraic partial action associated with an ultragraph}

In this subsection we show that the induced algebraic partial action arsing from the action defined in Subsection~\ref{s:partialaction} and the partial action defined in \cite{goncalves_royer_2019} are equivalent. We connect the necessary concepts from \cite{goncalves_royer_2019} with the concepts we developed so far, but we will refrain from reproducing here the whole construction given in \cite{goncalves_royer_2019}. We maintain the notation of the previous section.

For the remainder of the section $R$, is a commutative unital ring and, for a Hausdorff space $X$, $\Lc(X,R)$ denotes the set of locally constant $R$-valued functions on $X$ with compact support.

Given an ultragraph $\G$, we let
\begin{equation} \label{eqn:setY}
Y=\mathfrak{p}^{\infty}\cup\{(\alpha,v)\mid \alpha\in\awstar,\ v\in G^0_s\cap r(\alpha)\}.
\end{equation}

\begin{remark} \label{rem:sets.Y.and.X}
The set $Y$ is essentially the same as $X$ considered in \cite{goncalves_royer_2019}. The only difference is that we are using $(\eword,v)$ instead of $(v,v)$, recalling that $\eword$ is the empty word.
\end{remark}

For each $e=(\alpha,A,\alpha)\in E(S)\setminus\{0\}$, we define
\[Y_e=\{y\in Y \mid |y|>|\alpha|,\ y_{1,|\alpha|}=\alpha,\ s(y_{|\alpha|+1})\in A\}\cup\{(\alpha,v)\mid v\in A\cap G^0_s\}.\]
We also define $Y_0=\emptyset.$

\begin{remark} \label{rem:SameOpenSets}
Each set defined directly after \cite[Notation~3.4]{goncalves_royer_2019} can be seen as $Y_e$ for some $e\in E(S)$: for $\alpha,\beta\in\awstar$ and $A\in\G^0$, the set $X_{\alpha\beta^{-1}}$ corresponds to $Y_{(\alpha,r(\alpha)\cap r(\beta),\alpha)}$, $X_A$ corresponds to $Y_{(\eword,A,\eword)}$ and $X_{\alpha A}$ corresponds to $Y_{(\alpha,A,\alpha)}$.
\end{remark}

The following lemma is a simplification of \cite[Lemma~3.6]{goncalves_royer_2019}. 

\begin{lema} \label{lem:intersection.Y}
For every $e,f\in E(S)$, $Y_e\cap Y_f=Y_{ef}$.
\end{lema}

Let $\acfg{D}$ be the Boolean algebra in $\powerset{Y}$ generated by the family $\{Y_e\}_{e\in E(S)}$, and let $\widehat{\acfg{D}}$ be the the set of all ultrafilters on $\acfg{D}$, that is, the Stone dual of $\acfg{D}$. We will show that $\widehat{\acfg{D}}$ is homeomorphic to $\ftight$, but for this we need a few auxiliary results first.

For $n\in\N$ and $e,e_1,\ldots,e_n$, we define $Y_{e:e_1,\ldots,e_n}=Y_e\cap Y_{e_1}^c\cdots\cap Y_{e_n}^c$. Using Lemma \ref{lem:intersection.Y}, we can see that every element of $\acfg{D}$ is a finite union of sets of the form $Y_{e:e_1,\ldots,e_n}$.

\begin{lema}\label{lem:elements.Y.tight}
Let $n\in\N$ and $e,e_1,\ldots,e_n\in E(S)$ be given.
\begin{enumerate}[(i)]
    \item For $\alpha\in\mathfrak{p}^{\infty}$ and $\xi$ the only tight filter associated to $\alpha$, we have that $\xi\in V_{e:e_1,\ldots,e_n}$  if and only if $\alpha\in Y_{e:e_1,\ldots,e_n}$.
    
    \item For $(\alpha,v)\in Y$, where $\alpha\in\awstar$ and $v\in G^0_s$, let $\xi$ be the tight filter associated to $(\alpha,\usetr{\{v\}}{\acfra})$. Then, $\xi\in V_{e:e_1,\ldots,e_n}$ if and only if $(\alpha,v)\in Y_{e:e_1,\ldots,e_n}$.
\end{enumerate}
In particular, items (i) and (ii) define an injective map $\iota:Y\to\ftight$ such that $\iota(Y_{e:e_1,\ldots,e_n})=\iota(Y)\cap V_{e:e_1,\ldots,e_n}$.
\end{lema}

\begin{proof}
(i) Notice that for $e=(\beta,B,\beta)$ to be an element of $\xi$, we must have that $\beta$ is a beginning of $\alpha$ and $s(\alpha_{|\beta|+1})\in B$. The result then follows from the definitions of $V_{e:e_1,\ldots,e_n}$ and $Y_{e:e_1,\ldots,e_n}$.

(ii) In this case for $e=(\beta,B,\beta)$ to be an element of $\xi$, either $\beta$ is a beginning of $\alpha$ with $|\beta|<|\alpha|$ and $s(\alpha_{|\beta|+1})\in B$, or $\beta=\alpha$ and $v\in B$. Again, the result follows from the definitions of $V_{e:e_1,\ldots,e_n}$ and $Y_{e:e_1,\ldots,e_n}$.

The last part follows immediately from (i) and (ii).
\end{proof}

For $\mathbf{e}=(e,e_1,\ldots,e_n)\in E(S)^+$, recall that $V_{\mathbf{e}}=V_{e:e_1,\ldots,e_n}$ as in Section \ref{subsection:filters.E(S)}. We also define $Y_{\mathbf{e}}=Y_{e:e_1,\ldots,e_n}$.

\begin{lema}\label{lem:Y.dense.tight}
The map $\iota:Y\to\ftight$ given by Lemma \ref{lem:elements.Y.tight} has dense image.
\end{lema}

\begin{proof}
We prove that for every basic open set of the form $V_{\mathbf{e}}$, where $\mathbf{e}\in E(S)^+$, if $V_{\mathbf{e}}\neq\emptyset$, then there exists $y\in Y$ such that $\iota(y)\in V_{\mathbf{e}}$. Suppose then, that $\xia\in V_{\mathbf{e}}$ for some $\mathbf{e}\in E(S)^+$. If $|\alpha|=\infty$, then by Lemma \ref{lem:elements.Y.tight}, $\iota(\alpha)=\xia\in V_{\mathbf{e}}$.

Suppose now that $\xia$ is a filter of finite type and consider the three cases of Proposition~\ref{prop:tight.filters}. For case (i), we have that $\iota(\alpha,v)=\xia\in V_{\mathbf{e}}$ by Lemma~\ref{lem:elements.Y.tight}. For case (ii), suppose that $|\varepsilon(A)|=\infty$ for all $A\in\xi_{|\alpha|}$. Fix $A\in\xi_{|\alpha|}$ and assume without loss of generality that $\mathbf{e}=((\beta,B_0,\beta),(\alpha,B_1,\alpha),\ldots,(\alpha,B_k,\alpha),(\delta_1,C_1,\delta_1),\ldots,(\delta_l,C_l,\delta_l))$, where $\alpha=\beta\gamma$ for some $\gamma\in\awstar$, $r(B_0,\gamma)\in\xi_{|\alpha|}$ and for each $j=1,\ldots,l$, $\delta_j$ is either a beginning of $\alpha$ with $|\delta_j|<|\alpha|$, in which case $s(\alpha_{|\delta_j|+1})\notin C_j$, or $\delta_j$ is not a beginning of $\alpha$. Notice that $A':=(r(B_0,\gamma)\cap A)\setminus(B_1\cup\cdots\cup B_k)\in \xi_{|\alpha|}$ and therefore $|\varepsilon(A')|=\infty$. We can then choose $a\in \varepsilon(A')$ such that $a$ is not the $|\alpha|+1$ coordinate of any $\delta_j$. From $\alpha a$ we can build an element $y\in Y$ beginning with $\alpha a$ by adding edges until we find a sink or we add infinitely many edges and build an infinite pah. Then, by construction $y\in Y_{\mathbf{e}}$, and therefore by Lemma \ref{lem:elements.Y.tight}, $\iota(y)\in V_{\mathbf{e}}$.
Finally, for case (iii), if we let $A'$ as above, then $|A'\cap G^0_s|=\infty$. Analogous to case (ii), if we take any $v\in A' \cap G^0_s$, then $\iota(\alpha,v)\in V_{\mathbf{e}}$.
\end{proof}

\begin{lema}\label{lem:union.Y.and.V}
Let $\mathbf{e},\mathbf{f}^{(1)},\ldots,\mathbf{f}^{(m)}\in E(S)^+$, where $m\in\N$ with $m\geq 1$. If $Y_{\mathbf{e}}\scj \bigcup_{i=1}^m Y_{\mathbf{f}^{(i)}}$, then $V_{\mathbf{e}}\scj \bigcup_{i=1}^m V_{\mathbf{f}^{(i)}}$.
\end{lema}

\begin{proof}
Let $\xi\in V_{\mathbf{e}}$. By Lemma \ref{lem:Y.dense.tight}, there exists a sequence $\{y_n\}_{n\in \N}$ in $Y$ such that $\iota(y_n)$ converges to $\xi$. Since $V_{\mathbf{e}}$ is open, by Lemma \ref{lem:elements.Y.tight}, we may assume that $y_n\in Y_{\mathrm{e}}$ for all $n\in \N$. By the pigeonhole principle, there exists $i\in\{1,\ldots,m\}$ such that $Y_{\mathbf{f}^{(i)}}$ contains $y_n$ for infinitely many $n\in\N$. We can then take a subsequence $\{y_{n_k}\}_{k\in\N}$ such that $y_{n_k}\in Y_{\mathbf{f}^{(i)}}$ for all $k\in\N$. By Lemma \ref{lem:elements.Y.tight}, $\iota(y_{n_k})\in V_{\mathbf{f}^{(i)}}$ for all $k\in\N$. Since $V_{\mathbf{f}^{(i)}}$ is compact and $\ftight$ is Hausdorff, we conclude that $\xi\in V_{\mathbf{f}^{(i)}}$.
\end{proof}

For each $Z\in\acfg{D}$, let $\widehat{Z}=\{\ft\in\widehat{\acfg{D}}\mid Z\in\ft\}$. By the Stone duality, if $\acfg{K}(\widehat{\acfg{D}})$ is the family of all compact-open subsets of $\widehat{\acfg{D}}$, then the map $Z\in \acfg{D}\to \widehat{Z}\in \acfg{K}(\widehat{\acfg{D}})$ is an isomorphism of Boolean algebras.

We now prove that $ \widehat{\acfg{D}}$ and $ \ftight$ are homeomorphic.

\begin{proposicao}\label{prop:iso.dual.D.with.tight}
The map $\Phi:\widehat{\acfg{D}}\to\ftight$ given by
\[\Phi(\ft)=\{e\in E(S)\mid Y_e \in\ft\}\]
is a homeomorphism such that $\Phi(\widehat{Y_{\mathbf{e}}})=V_{\mathbf{e}}$ for all $\mathrm{e}\in E(S)^+$.
\end{proposicao}

\begin{proof}
We first check that $\Phi$ is well-defined. Given $\ft\in\widehat{\acfg{D}}$, since $Y_0=\emptyset$, we have that $Y_0\notin \ft$ and therefore $0\notin \Phi(\ft)$. That $\Phi(\ft)$ is a filter in $E(S)$ then follows from the fact that $e\leq f$ in $E(S)$ whenever $ef=e$ in $S$ and Lemma \ref{lem:intersection.Y}. By Theorem \ref{thm.filters.in.E(S)}, we can describe $\Phi(\ft)$ as a pair $(\alpha,\{\xi_n\}_{n=0}^{|\alpha|})$, where $\{\xi_n\}_{n=0}^{|\alpha|}$ is a complete family of filters for $\alpha$. We now prove, using Proposition~\ref{prop:ultrafilter}, that $\xi_n$ is an ultrafilter for every $n$. In order to do this, we fix $n$, we take an arbitrary $A\in \acfrg{\alpha_{1,n}}\setminus\xi_n$ and we show that there exists $C\in \xi_n$ such that $A\cap C=\emptyset$. Take $B\in\xi_n$ and notice that $A\cup B\in\xi_n$ and $A\cup B$ can be written as the disjoint union $A\cup (B\setminus A)$. Then $Y_{(\alpha_{1,n},A\cup B,\alpha_{1,n})}=Y_{(\alpha_{1,n},A,\alpha_{1,n})}\cup Y_{(\alpha_{1,n},B\setminus A,\alpha_{1,n})}$, again the union being disjoint. Since $\ft$ is an ultrafilter, it is a prime filter, and since $A\notin \xi_n$, we have that $Y_{(\alpha_{1,n},B\setminus A,\alpha_{1,n})}\in\ft$ and therefore $C=B\setminus A\in\xi_n$ is such that $A\cap C=\emptyset$. If $|\alpha|=\infty$, then $\Phi(\ft)$ is a tight filter by Theorem \ref{thm:TightFiltersType}. Suppose then that $|\alpha|<\infty$ and let us check Condition~(ii) of Theorem~\ref{thm:TightFiltersType}. If $A\in\xi_{|\alpha|}$ is such that $A\cap G^0_s=\emptyset$ and $\varepsilon(A)<\infty$, then $Y_{(\alpha,A,\alpha)}=\cup_{e\in\varepsilon(A)}Y_{(\alpha e,r(e),\alpha e)}$, and since $\ft$ is prime, there exists $e_0\in\varepsilon(A)$ such that $Y_{(\alpha e_0,r(e_0),\alpha e_0)}\in\ft$. This implies that $(\alpha e_0,r(e_0),\alpha e_0)\in\Phi(\ft)$, but this contradicts the fact that $\alpha$ is the labelled path associated to $\Phi(\ft)$. We conclude that $\Phi(\ft)\in\ftight$ when $|\alpha|<\infty$ and $|\alpha|=\infty$, showing $\Phi$ is well-defined.

For the inverse of $\Phi$, given $\xi\in\ftight$, we define
\[\ft_{\xi}=\{Z\in\acfg{D}\mid \exists\, \mathbf{e}\in E(S)^+\text{ such that }\xi\in V_{\mathbf{e}}\text{ and }Y_{\mathbf{e}}\scj Z\}.\]
In order to prove that $\ft_{\xi}$ is an ultrafilter, it suffices to show that for $n\in \N$ with $n\geq 2$, if $\mathbf{f}^{(1)},\ldots,\mathbf{f}^{(n)}\in E(S)^+$ are such that $\bigcup_{i=1}^m Y_{\mathbf{f}^{(i)}}\in\ft_{\xi}$, then there exists $i\in\{1,\ldots,n\}$ such that $Y_{\mathbf{f}^{(i)}}\in\ft_{\xi}$. By the definition of $\ft_{\xi}$, if $\bigcup_{i=1}^m Y_{\mathbf{f}^{(i)}}\in\ft_{\xi}$, then there exists $\mathbf{e}\in E(S)^+$ such that $\xi\in V_{\mathbf{e}}$ and $Y_{\mathbf{e}}\scj \bigcup_{i=1}^m Y_{\mathbf{f}^{(i)}}$. By Lemma \ref{lem:union.Y.and.V}, there exists $i\in\{1,\ldots,n\}$ such that $\xi\in V_{\mathbf{f}^{(i)}}$, but this implies that $Y_{\mathbf{f}^{(i)}}\in\ft_{\xi}$. The map $\Psi:\ftight\to\widehat{\acfg{D}}$ given by $\Psi(\xi)=\ft_{\xi}$ is then well-defined. Let us prove that $\Psi=\Phi^{-1}$.

Given $\ft\in\widehat{\acfg{D}}$ and $Z\in \ft$, there exists $\mathbf{e}=(e,e_1,\ldots,e_n)\in E(S)^+$ such that $Y_{\mathbf{e}}\in \ft$ and $Y_{\mathbf{e}}\scj Z$. Notice that $Y_e\in\ft$ and $Y_{e_i}\notin \ft$ for all $i=1,\ldots,n$. By definition, $e\in\Phi(\ft)$ and $e_i\notin\Phi(\ft)$ for all $i=1,\ldots,n$, which implies that $\Phi(\ft)\in V_{\mathbf{e}}$. We conclude that $Z\in \Psi(\Phi(\ft))$, and since $Z$ was arbitrary, $\ft\scj\Psi(\Phi(\ft))$. The equality follows from the fact that $\ft$ is an ultrafilter.

Now fix $\xi\in\ftight$. If $e\in \xi$, then $\xi\in V_e$, which implies that $Y_e\in\Psi(\xi)$, and hence $e\in\Phi(\Psi(\xi))$. We conclude that $\xi\scj\Phi(\Psi(\xi))$. For the reverse inclusion, suppose that $e\in \Phi(\Psi(\xi))$. In this case $Y_e\in\Psi(\xi)$, and therefore there exists $\mathbf{f}\in E(S)^+$ such that $\xi\in V_{\mathbf{f}}$ and $Y_{\mathbf{f}}\scj Y_e$. By Lemma \ref{lem:union.Y.and.V}, $V_{\mathrm{f}}\scj V_e$, which implies that $\xi\in V_e$, that is, $e\in \xi$. The equality $\xi=\Phi(\Psi(\xi))$ now follows.

The equality $\Phi(\widehat{Y_{\mathbf{e}}})=V_{\mathbf{e}}$ for all $\mathrm{e}\in E(S)^+$ follows from the construction of $\Phi$. Finally, the family $\{\widehat{Y_{\mathbf{e}}}\}_{\mathbf{e}\in E(S)^+}$ is a basis for the topology on $\widehat{\acfg{D}}$, which is sent by $\Phi$ to the basis $\{V_{\mathbf{e}}\}_{\mathbf{e}\in E(S)^+}$ of $\ftight$. Hence $\Phi$ is a homeomorphism.
\end{proof}


Next we prove that the algebra $D$ defined in \cite{goncalves_royer_2019} is isomorphic to $\Lc(\widehat{\acfg{D}},R) $.  Before we present this result we make a few observations.

If we define $\acfg{C}=\{Y_{(\alpha,A,\alpha)}\mid (\alpha,A,\alpha)\in E(S), A\in\G^0\}$, then the algebra $D$ defined in \cite{goncalves_royer_2019} coincides with $F_{\acfg{C}}(Y)$. Notice that the algebra of sets generated by $\acfg{C}$ is $\acfg{D}$, and thus, by Lemma~\ref{lem:boolean.algebra.function}, we have that $D=F_{\acfg{D}}(Y)$. Hence, 
\begin{equation} \label{dfn:D.genrating.set}
    D=\mathrm{span}\{1_{Y_{(\alpha,A,\alpha)}}: (\alpha,A,\alpha)\in E(S) \}. 
\end{equation}

\begin{proposicao}\label{prop:iso.lc.with.D}
There exists an isomorphism of $R$-algebras $\phi:\Lc(\widehat{\acfg{D}},R)\to D$ such that $\phi(1_{\widehat{Z}})=1_Z$ for all $Z\in\acfg{D}$.
\end{proposicao}

\begin{proof}
Given a non-zero element $f\in\Lc(\widehat{\acfg{D}},R)$, let $\{r_1,\ldots,r_n\}$ be the non-zero elements belonging to the image of $f$ and such that $r_i\neq r_j$ whenever $i\neq j$. For each $i=1,\ldots,n$, let $U_i=f^{-1}(r_i)$. Then $U_i$ is a compact-open subset of $\widehat{\acfg{D}}$ and therefore there is a unique $Z_i\in\acfg{D}$ such that $U_i=\widehat{Z_i}$. Define $\phi(f)=\sum_{i=1}^n r_i1_{Z_i}$ and $\phi(0)=0$. Using the isomorphism of Boolean algebra $Z\in \acfg{D}\to \widehat{Z}\in \acfg{K}(\widehat{\acfg{D}})$, it is easy to check that $\phi$ is a homomorphism of $R$-algebras. We have that $\phi$ is surjective because its image contains the generators of $D$. Moreover, observe that for $Z\in\acfg{D}$, we have that $1_Z=0\Leftrightarrow Z=\emptyset\Leftrightarrow \widehat{Z}=\emptyset\Leftrightarrow 1_{\widehat{Z}}=0$, from where we conclude that $\phi$ is injective.
\end{proof}

We have identified $\Lc(\widehat{\acfg{D}},R)$ with $\acfg{D}$ above, but our goal is to identify $\Lc(\ftight,R)$ with $ D$. We will do this using the identification between $\widehat{\acfg{D}}$ and $\ftight$ given in Proposition~\ref{prop:iso.dual.D.with.tight}. To this end, we need the following lemma, which describes a typical open set $V_t$, where $t\in \mathbb{F}$, that appears in the definition of the partial action $\Delta$ (see (\ref{eqn:labelled.top.partial.action})).
\begin{lema} \label{lem:open.set.alpha.beta}
For $\alpha,\beta \in\awstar $ we have that $V_{\alpha\beta^{-1}}= V_{(\alpha,r(\alpha)\cap r(\beta),\alpha)}$.
\end{lema}
\begin{proof} If $\xi\in V_{\alpha\beta^{-1}}$ then $r(\alpha)\cap r(\beta)\neq \emptyset$ (see \cite[Lemma~3.10]{MR4109095}), and by the partial action associated with a labelled space (\ref{eqn:labelled.top.partial.action}) \[\xi=\varphi_{\alpha\beta^{-1}}(\eta)\]
for some $\eta\in V_{(\beta\alpha^{-1})^{-1}}$. That is, $\eta=\eta^{\beta\gamma}$, for some $\gamma\in \awinf\cup \awstar$. Then, $(\eword,r(\beta),\eword)\in \varphi_{\beta^{-1}}(\eta) $, which implies that $(\alpha,r(\alpha)\cap r(\beta), \alpha)\in \varphi_{\alpha\beta^{-1}}(\eta)=\xi$, proving that $V_{\alpha\beta^{-1}}\subseteq V_{(\alpha,r(\alpha)\cap r(\beta),\alpha)}$.

For the reverse inclusion, if $\xi\in V_{(\alpha,r(\alpha)\cap r(\beta),\alpha)}$, then $(\alpha,r(\alpha)\cap r(\beta),\alpha)\in \xi$ and thus $\xi=\xi^{\alpha\gamma}$ for some $\gamma\in \awinf\cup \awstar$. Then, 
$(\eword,r(\alpha)\cap r(\beta),\eword)\in \varphi_\alpha^{-1}(\xi)$, which implies that $(r(\alpha)\cap r(\beta))\in \varphi_\alpha^{-1}(\xi)_{0}$. Since $\varphi_\alpha^{-1}(\xi)_{0}$ is an ultra filter, it follows that $r(\beta)\in\varphi_\alpha^{-1}(\xi)_{0}$. Thus, $\varphi_\alpha^{-1}(\xi)$ is in the set of tight filters that can be ``glued" to $\beta$. Now, let $\eta=\varphi_\beta(\varphi_\alpha^{-1}(\xi))$. Then, 
\[\xi=\varphi_\alpha(\varphi_\beta^{-1}(\eta)).\]
That is, $\xi\in V_{\alpha\beta^{-1}}$, and thus $V_{(\alpha,r(\alpha)\cap r(\beta),\alpha)}\subseteq V_{\alpha\beta^{-1}}$.
\end{proof}

For $t\in \mathbb{F}$, 
\[D_t=\mathrm{span}\{1_t 1_{Y_{(\alpha,A,\alpha)}}:(\alpha,A,\alpha)\in E(S)\}\]
defines an ideal in $D$. Note that if $t=\gamma\delta^{-1}$ for some $\gamma,\delta\in\awstar$, then $D_{\gamma\delta^{-1}}=D_{X_{\gamma\delta^{-1}}}$ in the notation of \cite{goncalves_royer_2019}, which corresponds to our $D_{Y_{(\gamma,r(\gamma)\cap r(\delta),\gamma)}}$ (see Remark~\ref{rem:SameOpenSets}). For the remainder of this section we retain the notation from earlier by still letting $\phi$ denote the isomorphism from Proposition~\ref{prop:iso.lc.with.D} and $\Phi$ the homeomorphism from Proposition~\ref{prop:iso.dual.D.with.tight}.
\begin{proposicao} \label{prop:isom.ideals}
     For any compact open set $U\subset \ftight$, we define 
     \[\psi(1_U)=\phi(1_{\Phi^{-1}(U)}).\]
     Then $\psi$ extends to an isomorphism of $\Lc(\ftight,R)$ onto $D$.
     Furthermore, if $t\in\mathbb{F}$ is in reduced form, then $\psi(\Lc(V_t,R))=D_t$. 
\end{proposicao}
  \begin{proof}
     The claim that $\psi$ extends to an isomorphism of $\Lc(\ftight,R)$ onto $D$ follows from  Proposition~\ref{prop:iso.dual.D.with.tight} and Proposition~\ref{prop:iso.lc.with.D}. 
     
    Let  $t\in\mathbb{F}$ be in reduced form. If $t=\eword$, then $\Lc(V_\eword,R)=\Lc(\ftight,R)$, and we are done. If $t\neq\eword$ and $V_t\neq\emptyset$, then $t\in \{\alpha\mid\alpha\in\awplus\}\cup\{\alpha^{-1}\mid\alpha\in\awplus\}\cup\{\alpha\beta^{-1}\mid \beta,\alpha\in\awplus, r(\alpha)\cap r(\beta)\neq\emptyset\}$, by \cite[Lemma 3.11]{MR4109095}. Similarly, by \cite[Notation 3.4]{goncalves_royer_2019}, if $D_t\neq\emptyset$, then $t=\alpha\beta^{-1}$ with $r(\alpha)\cap r(\beta)\neq \emptyset$. Hence we may assume without loss of generality that $t=\alpha\beta^{-1}$ with $\alpha,\beta\in\awstar$  and  $r(\alpha)\cap r(\beta)\neq \emptyset$. 
    
    We first show that $\psi(\Lc(V_t,R))\subseteq D_t$. By Lemma~\ref{lem:open.set.alpha.beta} we have that $V_t=V_{\alpha\beta^{-1}}=V_{(\alpha, r(\alpha)\cap r(\beta),\alpha)}$, and then it follows from Lemma~\ref{prop:iso.dual.D.with.tight} that $\psi(1_{V_{(\alpha, r(\alpha)\cap r(\beta),\alpha)}})=1_{Y_{(\alpha, r(\alpha)\cap r(\beta),\alpha)}}$. Let $U\subseteq V_{t}$ be a compact open set. Then $\Phi^{-1}(U)$ is a compact open set in $\widehat{\acfg{D}}$. That is, $\Phi^{-1}(U)\in \acfg{K}(\widehat{\acfg{D}})$. Hence there is exists a $Z\in \acfg{D}$ such that $\Phi^{-1}(U)=\widehat{Z}\subset \widehat{Y}_{(\alpha, r(\alpha)\cap r(\beta),\alpha)}$.  Then, 
    \begin{eqnarray*}
     \psi(1_{U} )&=& \psi(1_{U}1_{V_{(\alpha, r(\alpha)\cap r(\beta),\alpha)}}) \\
     &=& 1_Z 1_{Y_{(\alpha, r(\alpha)\cap r(\beta),\alpha)}},
    \end{eqnarray*}
    which belongs to $D_t$. If $f\in \Lc(V_t,R)$ is arbitrary, then $f=\sum_{i=1}^{n}r_i 1_{U_i}$, where $r_i\in R$ and $U_i\subseteq V_t$ are compact open for each $i=1,2,\ldots,n$. Then, $\psi(f)=\sum_{i=1}^n r_i \psi(1_{U_{i}}) \in D_t$. Hence $\psi(\Lc(V_t,R))\subseteq D_t$.
    
    To show that $\psi(\Lc(V_t,R))\subseteq D_t$, it will suffice to show that each generator of $D_t$ is the image of some function from $\Lc(V_t,R)$. Consider $1_{Y_{(\alpha, r(\alpha)\cap r(\beta),\alpha)}}1_{Y_{(\gamma,C,\gamma)}}\in D_t$, with $(\gamma,C,\gamma)\in E(S)$. Then 
    \[\psi({1_{V_{(\alpha, r(\alpha)\cap r(\beta),\alpha)}}1_{V_{(\gamma,C,\gamma)}}})=1_{Y_{(\alpha, r(\alpha)\cap r(\beta),\alpha)}}1_{Y_{(\gamma,C,\gamma)}}, \]
    and $1_{V_{(\alpha, r(\alpha)\cap r(\beta),\alpha)}}1_{V_{(\gamma,C,\gamma)}}\in  \Lc(V_{(\alpha,r(\alpha)\cap r(\beta),\alpha)},R)=\Lc(V_{\alpha\beta^{-1}},R)$. Hence, $\psi$ maps $\Lc(V_t,R)$ onto $D_t$.
  \end{proof}

Let $\Delta=(\{V_t\}_{t\in\mathbb{F}}, \{\varphi_t\}_{t\in\mathbb{F}})$ be the partial action on $\ftight$ defined in (\ref{eqn:labelled.top.partial.action}). Then $\Delta$ induces a dual algebraic partial action
$\hat{\Delta}= (\{\Lc(V_t,R)\}_{t\in\mathbb{F}}, \{\hat{\varphi_t}\}_{t\in\mathbb{F}})$, and each $\hat{\varphi_t}:\Lc(V_{t^{-1}},R)\to \Lc(V_t,R)$ is the isomorphism defined by $\hat{\varphi_t}(f)(\xi)=f\circ\varphi_{t^{-1}}(\xi)$.


We now state the main result of this section.

\begin{theorem} \label{prop:equivalent.partial.actions}
Let $\G$ be an ultragraph, $\Theta=(\{D_t\}_{t\in\mathbb{F}}, \{\theta_t\}_{t\in\mathbb{F}})$ be the algebraic partial action of $\mathbb{F}$ on $D$ as defined in \cite[Remark 3.8]{goncalves_royer_2019} and $\hat{\Delta}=(\{\Lc(V_t,R)\}_{t\in\mathbb{F}}, \{\hat{\varphi_t}\}_{t\in\mathbb{F}})$ be the dual of the topological partial action $\Delta=(\{V_t\}_{t\in\mathbb{F}}, \{\varphi_t\}_{t\in\mathbb{F}})$ defined in (\ref{eqn:labelled.top.partial.action}). Then $\Theta$ and $\hat{\Delta}$ are equivalent. 
\end{theorem}

\begin{proof}
If $\psi$ is the map of Proposition~\ref{prop:isom.ideals}, then $\psi\circ\hat{\varphi_t}=\theta_t\circ\psi$ for all $t\in\mathbb{F}$, from where the result follows.
\end{proof}

\begin{remark}
In the context of Leavitt path algebras the purely algebraic partial action defined in \cite{grskew} is shown to be equivalent to a topological partial action in \cite[Page 3964]{Hazgraded}. 
The above result generalizes this for ultragraph Leavitt path algebras.
\end{remark}

\section{Ultragraph algebras}\label{repolho}

In this section we focus on the realization of the algebras associated to ultragraphs as groupoid algebras. In the algebraic setting we first do the fundamental task of reconciling the two running definitions of an ultragraph Leavitt path algebra, and then we realize such algebras as Steinberg algebras. In the C*-algebraic context, we provide a description of a general ultragraph C*-algebra as a groupoid algebra, thereby generalizing \cite{MR2457327} (where ultragraphs are assumed to have no sinks) and  \cite{tasca2020kms} (where ultragraphs are assumed to satisfy Condition~(RFUM2)). Our results also provide a description of a general  ultragraph C*-algebra as a partial crossed product.

\subsection{Ultragraph Leavitt path algebra} \label{sec:Leavitt.path.algs.isom}

In this section we realize an ultragraph  Leavitt path algebra as a Steinberg algebra. However, as described in the introduction, it is of fundamental importance to first reconcile the two running definitions of ultragraph  Leavitt path algebras. This is done in Proposition~\ref{prop:isom.leavitt.path.algs}. After that we turn our attention to showing that the Leavitt path algebra of an ultragraph is isomorphic to the Steinberg algebra associated to the groupoid described in Subsection~\ref{vento}.

The definition of an ultragraph Leavitt path algebra that we will adopt is the following.

\begin{definicao}\label{def of ultragraph algebra}
Let $\mathcal{G}$ be an ultragraph and $R$ be a unital commutative ring. The Leavitt path algebra of $\mathcal{G}$, denoted by $L_R(\mathcal{G})$, is the universal $R$-algebra with generators $\{s_e,s_e^*:e\in \mathcal{G}^1\}\cup\{p_A:A\in \mathcal{G}^0\}$ and relations
\begin{enumerate}
\item $p_\emptyset=0,  p_Ap_B=p_{A\cap B},  p_{A\cup B}=p_A+p_B-p_{A\cap B}$, for all $A,B\in \mathcal{G}^0$;
\item $p_{s(e)}s_e=s_ep_{r(e)}=s_e$ and $p_{r(e)}s_e^*=s_e^*p_{s(e)}=s_e^*$ for each $e\in \mathcal{G}^1$;
\item $s_e^*s_f=\delta_{e,f}p_{r(e)}$ for all $e,f\in \mathcal{G}$;
\item $p_v=\sum\limits_{s(e)=v}s_es_e^*$ whenever $0<\vert s^{-1}(v)\vert< \infty$.
\end{enumerate}
\end{definicao}

As we mentioned before, the difference in the definitions of an ultragraph Leavitt path algebra lies in how the set of generalized vertices are defined.  Given an ultragraph $\G$,  let $\acf$ and $\G^0$ be as in Definition~\ref{prop_vert_gen}, and recall that $\G^0$ is not necessarily closed under relative complements. We denote by $L_R(\G_r)$ the Leavitt path algebra associated with $\G$ by allowing $A,B\in\acf$ in 1. of Definition~\ref{def of ultragraph algebra}, that is, $L_R(\G_r)$ is the algebra as defined in \cite[Definition~2.1]{imanfar2017leavitt}.

\begin{proposicao} \label{prop:isom.leavitt.path.algs} 
Let $\G$ be an ultragraph. Then $L_R(\G_r)$ and $L_R(\G)$ are isomorphic.
\end{proposicao}
\begin{proof}
In $L_R(\G_r)$, the family $\{p_A, e, e^*:A\in \G^0, e\in \G^1\}$ satisfies the relations defining $L_R(\G)$. By universality this gives us a homomorphism $\pi: L_R(\G)\rightarrow L_R(\G_r)$. 
We build the inverse of this homomorphism by describing a family $\{\widetilde{p}_A, e, e^*:A\in \acf, e\in \G^1\})$ inside $L_R(\G)$ satisfying the relations defining $L_R(\G_r)$.

Let $L_{\G^0}$ be the algebra as in Definition~\ref{def:univ.algebra.lattice.sets}, taking $\G^0$ as $\acfg{C}$  and $\{q_A\}_{A\in\G^0}$ as the generators. By the universal property of $L_{\G^0}$, there exists an homomorphism $\phi:L_{\G^0}\to L_R(\G)$ such that $\phi(q_A)=p_A$ for all $A\in\G^0$. By Proposition~ \ref{prop:univ.algebra.lattice.sets}, there exists an isomorphism $\psi:F_{\G^0}(G^0)\to L_{\G^0}$ such that $\psi(1_A)=q_A$ for all $A\in\G^0$, where $F_{\G^0}(G^0)$ is the subalgebra of $F(G^0)$ generated by $\{1_C:C\in \G^0\}$. By Lemma~\ref{lem:boolean.algebra.function}, we have that $F_{\G^0}(G^0)=F_{\acf}(G^0)$. Now, for each $A\in\acf$, we define \[\widetilde{p}_A:=\phi(\psi(1_A)).\] Clearly $\widetilde{p}_A=p_A$ for all $A\in\G^0$. Also, it is readily checked that $\widetilde{p}_A\widetilde{p}_B=\widetilde{p}_{A\cap B}$ and $\widetilde{p}_{A\cup B}=\widetilde{p}_A+\widetilde{p}_B-\widetilde{p}_{A\cap B}$ for all $A,B\in\acf$.

Notice that the family $\{\widetilde{p}_A, e, e^*:A\in \acf, e\in \G^1\}$ inside $L_R(\G)$ satisfies the relations defining $L_R(\G_r)$ and hence, by universality, we obtain a  homomorphism $\theta:L_R(\G_r) \rightarrow L_R(\G)$. To finish the proof we have to show that $\pi$ and $\theta$ are inverses of each other. 
It is straightforward that $\theta \circ \pi = id$. We verify that $ \pi \circ \theta =id $ on the generators of $L_R(\G_r)$.  Note first that $\pi \circ \theta (e) = e$ and $\pi \circ \theta (e^*) = e^*$ for all $e \in \G^1$. Let $B \in \acf$. Then $\theta(p_B)= \widetilde{p}_B = \phi(\psi(1_B))$. By Lemma~\ref{lem:boolean.algebra.function}, with $\acfg{C} = \G^0$, we have that $1_B \in F_{\G^0}(G^0)$. So we can write $1_B$ as a linear combination of the form $1_B=\sum_{i=1}^n c_i 1_{B_i} $, where $B_i\in \G^0$.  Hence $\pi\circ \theta (p_B) =\sum_{i=1}^n c_i p_{B_i}$. Now note that, similarly to what was done for $\G^0$ above, we can find a homomorphism from $F_{\acf}(G^0)$ to $L_R(\G_r)$ sending $1_A$ to $p_A$ for every $A\in\acf$. Applying this homomorphism to the equation $1_B=\sum_{i=1}^n c_i 1_{B_i} $ we conclude that $p_B=\sum_{i=1}^n c_i p_{B_i}$ inside $L_R(\G_r).$ So $\pi\circ \theta (p_B) = p_B$ as desired.
\end{proof}


Although the algebras obtained with the different versions of generalized vertices agree, in some situations there are relevant differences depending on the definition used. For example, the graded uniqueness theorem for ultragraph Leavitt path algebras proved in \cite{imanfar2017leavitt} gives conditions when a homomorphism from an ultragraph Leavitt path algebra is injective. Among the conditions, one is required to check that the homomorphism does not vanish on the generalized vertices. Of course if the set of generalized vertices is larger, it is, a priori, harder to verify the condition. As we show below though,  this is not the case, that is, it is enough to check the condition on the set of generalized vertices defined without use of relative complements, that is, on $\G^0$. Furthermore, from this we clearly get the 
the graded uniqueness theorem for $L_R(\G)$ (by composing homomorphisms), which we state below for completeness.

\begin{lema}Suppose  that $\phi: L_R(\G_r) \rightarrow S$ is a homomorphism that does not vanish in $\G^0$. Then $\phi$ does not vanish in $\acf$.
\end{lema}
\begin{proof}
Let $B\in \acf$ and take $v\in B$. Since $v\in \G^0$ we have that $\phi(v) \neq 0$. Hence $0\neq \phi (v) = \phi(v B) = \phi(v)\phi(B)$. We conclude that $\phi(B) \neq 0$.
\end{proof}

\begin{teorema}\label{acara}(c.f. \cite[Corollary~2.18]{imanfar2017leavitt})
Let $\G$ be an ultragraph, $R$ be a unital commutative ring and $S$ be a $\Z-$graded ring. If $\pi:L_R(\G)\rightarrow S$ is a graded ring homomorphism such that $\pi(r p_A) \neq 0$ for all non-empty $A\in \G^0$ and all nonzero $r\in R$, then $\pi$ is injective.
\end{teorema}


Given an ample groupoid $\Gamma$, we denote by $A_R(\Gamma)$ the groupoid algebra, known as Steinberg algebra, defined in \cite{BenGroupoid}. In the next result, we give a realization of the Leavitt path algebra of an ultragraph as a Steinberg algebra.

\begin{teorema}\label{kart}
 Let $\G$ be an ultragraph. Then, there is an isomorphism $\kappa:L_R(\G)\to A_R(\mathbb{F} \ltimes_{\Delta} \ftight)$ given on the generators of $L_R(\G)$ by 
\begin{equation}\label{eqn:Leavitt.Steinb.isom}
    \begin{split}
        & \kappa(p_A)=  1_{\{\eword\}\times V_{(\eword,A,\eword)}}  \\
        & \kappa(s_e)= 1_{\{e\} \times V_{(e,r(e),e)}}, \text{ and } \\
        & \kappa(s_{e^*})= 1_{\{e^{-1}\}\times V_{(\eword,r(e),\eword)}} \\
    \end{split}
\end{equation}
  for each $A\in \mathcal{G}^{0}$ and $e\in\mathcal{G}^1$.
\end{teorema}
\begin{proof}
Let $\mathbb{F} \ltimes_{\Delta} \ftight$ denote the transformation groupoid associated with $\Delta$ in (\ref{dfn:trans.groupoid}). Then $\mathbb{F} \ltimes_{\Delta} \ftight$ is an ample Hausdorff groupoid, \cite[Lemma 5.4]{MR4109095}. By \cite[Theorem 3.2]{MR3743184} the partial skew group ring $\Lc(\ftight,R)\rtimes_{\hat{\Delta}} \mathbb{F} $ is isomorphic to $A_R(\mathbb{F} \ltimes_{\Delta} \ftight)$.

By \cite[Theorem 3.10]{goncalves_royer_2019} the Leavitt path algebra $L_R(\mathcal{G})$ associated with $\mathcal{G}$ is isomorphic to the partial skew ring $D \rtimes_{\Theta} \mathbb{F}$. But, by Theorem~\ref{prop:equivalent.partial.actions}, the partial actions $\Theta$ and $\hat{\Delta}$ are equivalent. Hence, $D \rtimes_{\Theta} \mathbb{F}$ is isomorphic to  $\Lc(\ftight,R)\rtimes_{\hat{\Delta}} \mathbb{F}$, from which it follows that $L_R(\mathcal{G})$ is isomorphic to $A_R(\mathbb{F} \ltimes_{\Delta} \ftight)$. The equations (\ref{eqn:Leavitt.Steinb.isom}) are obtained by composing the two isomorphisms mentioned above.  
\end{proof}

As the composition of two isomorphisms, $\kappa$ in Theorem~\ref{kart} factors through a partial skew ring  $\Lc(\ftight,R)\rtimes_{\hat{\Delta}} \mathbb{F}$. Although $\kappa$ is explicitly given on a generating set of $L_R(\G)$, the dynamics in the partial skew ring plays a big part in how $\kappa$ maps more general elements. In Section~\ref{sec:cores} we are interested in elements of the from $\kappa(s_{\alpha}p_As_{\beta^{*}})$, with $\alpha, \beta\in \awstar$ and $A\in\G^0$. To see what $\kappa(s_{\alpha}p_As_{\beta^{*}})$ looks like in $A_R(\mathbb{F} \ltimes_{\Delta} \ftight)$, it is helpful to first compute the product in the partial skew ring, and then map into $A_R(\mathbb{F} \ltimes_{\Delta} \ftight)$. That is, 
\[s_{\alpha}p_As_{\beta^{*}}\mapsto (1_{V_{(\alpha,r(\alpha),\alpha)}}\delta_\alpha)(1_{V_{(\eword,A,\eword)}}\delta_\eword)(1_{V_{(\eword,r(\beta),\eword)}}\delta_{\beta^{-1}})\in \Lc(\ftight,R)\rtimes_{\hat{\Delta}} \mathbb{F},\]
where, for example,  $\delta_{\alpha}$ is merely a placeholder indicating that $1_{V_{(\alpha,r(\alpha),\alpha)}}$ belongs to the ideal $E_{\alpha}$. Now, computing inside the partial skew ring yields 
\[ (1_{V_{(\alpha,r(\alpha),\alpha)}}\delta_\alpha)(1_{V_{(\eword,A,\eword)}}\delta_\eword)(1_{V_{(\eword,r(\beta),\eword)}}\delta_{\beta^{-1}}) = 1_{V_{(\alpha,A\cap r(\alpha)\cap r(\beta),\alpha)}}\delta_{\alpha\beta^{-1}} \]
(see for example \cite[Equation~(4.2)]{MR4109095}). Then, 
\begin{equation} \label{eqn:isom.general.products}
    \kappa(s_{\alpha}p_As_{\beta^{*}})= 1_{\{\alpha\beta^{-1}\}\times V_{(\alpha,A\cap r(\alpha)\cap r(\beta),\alpha)}}.
\end{equation}

\subsection{The C*-algebra of an ultragraph}

In this section, we show how the C*-algebra of an arbitrary ultragraph defined by Tomforde in \cite{MR2050134} can be written as a groupoid C*-algebra and as a partial crossed product, generalizing the results in \cite{MR3938320,MR2457327,tasca2020kms}.

\begin{definicao}(\cite{MR2050134})
Let $\mathcal{G}$ be an ultragraph. The \emph{C*-algebra associated to $\G$}, denoted by $C^*(\mathcal{G})$, is the universal C*-algebra generated by a collection of mutually orthogonal partial isometries $\{s_e:e\in \mathcal{G}^1\}$ and projections $\{p_A:A\in \mathcal{G}^0\}$, subject to the relations
\begin{enumerate}
\item\label{item:projection} $p_\emptyset=0,  p_Ap_B=p_{A\cap B},  p_{A\cup B}=p_A+p_B-p_{A\cap B}$, for all $A,B\in \mathcal{G}^0$;
\item $s_e^* s_e= p_{r(e)}$ for all $e\in \G^1$;
\item $s_e s_e^* \leq p_{s(e)}$ for all $e\in \G^1$;
\item $p_v=\sum\limits_{s(e)=v}s_es_e^*$ whenever $0<\vert s^{-1}(v)\vert< \infty$.
\end{enumerate}
\end{definicao}

As with Leavitt path algebras, we can define another C*-algebra by allowing the relations in \ref{item:projection} to be valid for all $A,B\in\acf$. That we still get $C^*(\G)$ with these extra relations is proven in \cite{MR2413313}.

For a labelled space $\lspace$, there is also a definition of a C*-algebra $C^*\lspace$ such that if $\lspace$ is the labelled space associated to $\G$, the $C^*\lspace\cong C^*(\G)$ \cite[Example 2]{MR3614028}. Let also $\Gamma(\ftight,\sigma)$ be the groupoid defined in Section \ref{groupoids} and $\Delta$ the partial action defined in Section \ref{s:partialaction}. Then $\Delta$ induces a C*-algebraic partial action and we may form the partial crossed product $C_0(\ftight)\rtimes_{\Delta} \mathbb{F}$.

\begin{theorem}  \label{thm:isom.cstar.algs}
    Let $\G$ be an arbitrary ultragraph, then
    \[C^*(\G)\cong C^*(\Gamma(\ftight,\sigma))\cong C^*(\mathbb{F} \ltimes_{\Delta}\ftight) \cong C_0(\ftight) \rtimes_{\Delta} \mathbb{F}.\]
\end{theorem}
\begin{proof}
The first isomorphism follows from \cite[Example 2]{MR3614028} and \cite[Theorems 3.7 and 5.8]{Gil3}, the second isomorphism follows from \cite[Theorem 5.5]{MR4109095} and the third isomorphism follows from \cite[Theorem 3.3]{MR2045419}.
\end{proof} 

We remark that the C*-algebra $C_0(\ftight)\rtimes_{\Delta} \mathbb{F}$ is generated by 
\[ \{1_{V_{(\eword,A,\eword)}}\delta_\eword,1_{V_{(a,r(a),a)}}\delta_a: A\in \acf, a\in \alf\}.\]
Let $\kappa_*:C^*(\G)\to C^*(\mathbb{F}\rtimes_{\Delta}\ftight)$ denote the isomorphism in Theorem~\ref{thm:isom.cstar.algs}. Then $\kappa_*$ is given on the generators of  $C^*(\G)$ by 
\begin{equation} \label{eqn:isom.cstar.generators}
  \begin{split}
        & \kappa_*(p_A)= 1_{\{\eword\}\times V_{(\eword,A,\eword)}} \\
        & \kappa_*(s_e)=1_{\{e\}\times V_{(e,r(e),e)}}. 
    \end{split}
\end{equation}

\begin{remark}\label{rmk:MM.groupoid}
As pointed out in Remark \ref{rmk:ultraset}, the description of the ultragraph groupoid $\mathfrak{G}_{\G}$ in \cite{MR2457327} is incomplete, which implies that the isomorphism in \cite[Theorem 22]{MR2457327} does not always hold. For instance, in Example \ref{picapau}, there is a sequence in $V_{(e_0,r(e_0),e_0)}$ that converges to an element $\xi$, which corresponds to a pair $(e_0,\ft)$, where $\ft$ is not a principal filter. This implies that the set $\mathcal{A}'((e_0,r(e_0)),e_0)$ described in \cite{MR2457327} is not actually compact because it contains a sequence with no convergent subsequence, and therefore the characteristic function of this set is not an element of $C^*(\mathfrak{G}_{\G})$. This means that the map from $C^*(\G)$ to $C^*(\mathfrak{G}_{\G})$ that sends $s_{e_0}$ to $1_{\mathcal{A}'((e_0,r(e_0)),e_0)}$ (as in \cite{MR2457327}) is not well-defined.
\end{remark}

\section{Abelian core subalgebras and the generalized uniqueness theorems} \label{sec:cores}

In this section we prove generalized uniqueness theorems for ultragraph algebras in both the analytical and algebraic setting. These uniqueness theorems have the advantage of not requiring an aperiodicity nor a gauge invariance nor a graded homomorphism assumption. By identifying the abelian core subalgebras we also answer, in the context of ultragraph algebras, a question raised in \cite{GILCANTO2018227} for Leavitt path algebras, namely we characterize the ultragraph Leavitt path algebras such that the center is equal to the core.

For the results of this section, we will use the isotropy bundle of a groupoid. Recall that if $G$ is a groupoid with source map $s$ and range map $r$, the isotropy bundle is given by $\text{Iso}(G)=\{\gamma\in G\mid s(\gamma)=r(\gamma)\}$. For more details, we refer the reader to \cite{BNRSW,CEP}. To fix the notation, $\text{Iso}(G)^0$ represents the interior of the isotropy bundle of a groupoid $G$.

 We recall below the generalized graded uniqueness theorem for Steinberg algebras (see also \cite{CEM}).
 
 \begin {teorema}[Generalized Uniqueness Theorem]\cite[Theorem~3.1]{CEP} 
\label {thm:Uniqueness}
Let ${\mathcal H}$ be a second-countable, ample, Hausdorff groupoid and
let $R$ be a unital commutative ring.
Suppose that $A$ is an $R$-algebra and that $\pi :A_R({\mathcal H}) \to A$ is a ring homomorphism.
Then $\pi $ is injective if and only if $\pi \circ \iota $ is injective, where $\iota$ is the natural inclusion of $A_R(\text{Iso}({\mathcal H})^0)$ in $  A_R({\mathcal H})$.
\end {teorema}

In the context of groupoid C*-algebras, the analogous result is the following:

\begin{teorema}\label{thm:uniqueness}\cite[Theorem~3.1 (b)]{BNRSW}
  Let $G$ be a locally compact Hausdorff \'etale groupoid. If $\pi : C^*_r(G) \to D$ is a
    $C^*$-homomorphism, then $\pi$ is injective if and only if
    $\pi \circ \iota_r$ is an injective homomorphism of
    $C^*_r(\text{Iso}(G)^0))$,  where $ \iota_r : C^*_r(\text{Iso}(G)^0)) \to C^*_r(G)$ is the inclusion map.
\end{teorema}

We want to use the above uniqueness results in combination with our characterization of ultragraph algebras as groupoid algebras to obtain generalized uniquess theorems for ultragraph algebras. This motivates our definition of the abelian core of an ultragraph algebra, which we present after we recall some relevant concepts below. 

Let $\G$ be an ultragraph and let $R$ be a unital commutative ring. Recall that $L_R(\G)=\mathrm{span}_R\{s_\alpha p_As_{\beta^*}: \text{$\alpha$, $\beta$ are paths, and $A\in\G^0$}\}$ and if $R=\mathbb{C}$, then for $C^*(\G)$ the same holds by taking the closure of the span on the right side (see \cite{imanfar2017leavitt} and \cite{MR2050134}). Denote the set of generators of the algebra by $G_\G$, that is,

\[G_\G = \{s_\alpha p_A s_{\beta^*}: \text{$\alpha$, $\beta$ are paths, $A\in\G^0$, and $r(\alpha)\cap A \cap r(\beta)\neq \emptyset$\}.}\]

\begin{definicao}
Let $\G$ be an ultragraph and let $R$ be a unital commutative ring. The diagonal subalgebra $D(L_R(\G))$  of $L_R(\G)$ (respectively $D(C^*(\G))$ of $C^*(\G)$) is the $R-$subalgebra (respectively $C^*$-subalgebra) generated by elements of $G_\G$ such that $\alpha=\beta$.  
The commutative (abelian) core of $L_R(\G)$ (respectively of $C^*(\G)$) is the subalgebra $M(L_R(\G))$ (respectively  the $C^*$-subalgebra $M(C^*(\G))$) generated by elements of $G_\G$ that satisfy:
\begin{enumerate}
\item $\alpha = \beta$;
\item  $\alpha = \beta \lambda_{\beta}$ and $\lambda_{\beta}$ is a loop without exits;
\item  $\beta = \alpha \lambda_{\alpha}$ and $\lambda_{\alpha}$ is a loop without exits.
\end{enumerate}
We denote by $G^M_\G$ the set of all elements in $G_\G$ that satisfy one of the three above conditions.
\end{definicao}

\begin{remark}
The word \emph{core} is overused in the context of ultragraph algebras. The reader should not confuse the abelian core defined above with the core subalgebra defined in \cite{MR3856223}.
\end{remark}

Our next goal is to identify the abelian core of an ultragraph algebra with the algebra of the interior of the isotropy.  For this it will be convenient to use $\mathbb{F} \ltimes_{\Delta} \ftight$ as the groupoid associated with $\G$. We start by identifying $\text{Iso}(\mathbb{F}\ltimes_{\Delta} \ftight)^0$, but for this we need a couple of auxiliary results first.

\begin{lema}\label{lem:isolated.vs.exits}
An element $\xi\in\ftight$ with associated path $\beta\gamma^{\infty}$, for some path $\beta$ and some loop $\gamma$, is isolated if and only if $\gamma$ has no exits. In this case, $V_{(\beta\gamma^n,\{s(\gamma)\},\beta\gamma^n)}=\{\xi\}$ for all $n\in\N$.
\end{lema}

\begin{proof}
Suppose first that $\gamma$ has an exit. Then for any open neighborhood $V$ of $\xi$, for $n$ sufficiently large, $\xi\in V_{(\beta\gamma^n,\{s(\gamma)\},\beta\gamma^n)}\scj V$. We can then use the exit to build an element of $\ftight$ different from $\xi$ which is in $V_{(\beta\gamma^n,\{s(\gamma)\},\beta\gamma^n)}$, so that $\xi$ is not isolated.

Now suppose that $\gamma$ has no exits. This means that for all $i=1,\ldots,|\gamma|$ we have that $s^{-1}(s(\gamma_i))=\gamma_i$, which then implies that $V_{(\beta\gamma^n,\{s(\gamma)\},\beta\gamma^n)}=\{\xi\}$ for all $n\in\N$, and in particular $\xi$ is an isolated point.
\end{proof}

\begin{lema} \label{lem:interior.isotropy}
For $(t,\xi)\in \mathbb{F} \ltimes_{\Delta} \ftight$ such that $t\neq\eword$, we have that $(t,\xi)\in \text{Iso}(\mathbb{F} \ltimes_{\Delta} \ftight)^0$ if, and only if, there exist a path $\beta$ and a loop without exits $\gamma$ such that $t=\beta\gamma\beta^{-1}$ or $t=\beta\gamma^{-1}\beta^{-1}$, and $\xi$ is the element of $\ftight$ associated to $\beta\gamma^{\infty}$. In this case $(t,\xi)$ is an isolated point of $\text{Iso}(\mathbb{F} \ltimes_{\Delta} \ftight)^0$.
\end{lema}

\begin{proof}
Suppose first that $(t,\xi)\in \text{Iso}(\mathbb{F} \ltimes_{\Delta} \ftight)^0$. By \cite[Remark~6.11]{MR4109095}, there exist a path $\beta$ and a loop $\gamma$ such that $t=\beta\gamma\beta^{-1}$ or $t=\beta\gamma^{-1}\beta^{-1}$, and the associated path of $\xi$ is  $\beta\gamma^{\infty}$. By Proposition \ref{prop:tight.filters}, $\xi$ is the unique element of $\ftight$ with associated path $\beta\gamma^{\infty}$ so that $(\{t\}\times\ftight)\cap \text{Iso}(\mathbb{F} \ltimes_{\Delta} \ftight)=\{(t,\xi)\}$. Since $(t,\xi)\in \text{Iso}(\mathbb{F} \ltimes_{\Delta} \ftight)^0$, there exists an open set $U$ of $\mathbb{F} \ltimes_{\Delta} \ftight$ such that $(t,\xi)\in U\scj \text{Iso}(\mathbb{F} \ltimes_{\Delta} \ftight)$. This implies that $\{(t,\xi)\}=(\{t\}\times\ftight)\cap U$ is open and so is its projection in the second coordinate $\{\xi\}$. By Lemma \ref{lem:isolated.vs.exits}, $\gamma$ has no exits.

Suppose now that there exist a path $\beta$ and a loop without exits $\gamma$ such that $t=\beta\gamma\beta^{-1}$ or $t=\beta\gamma^{-1}\beta^{-1}$, and $\xi$ is the element of $\ftight$ associated to $\beta\gamma^{\infty}$. Notice that $\varphi_t(\xi)=\xi$ so that $(t,\xi)\in\text{Iso}(\mathbb{F} \ltimes_{\Delta} \ftight)$. By Lemma~\ref{lem:isolated.vs.exits}, $\{\xi\}$ is open in $\ftight$, and hence $\{(t,\xi)\}$ is an open neighborhood of $(t,\xi)$ inside $\text{Iso}(\mathbb{F} \ltimes_{\Delta} \ftight)$, from where the result follows.
\end{proof}

Since the unit space $(\mathbb{F} \ltimes_{\Delta} \ftight)^{(0)}$ of $\mathbb{F} \ltimes_{\Delta} \ftight$ is identified with $\ftight$, it follows that $A_R((\mathbb{F} \ltimes_{\Delta} \ftight)^{(0)})\cong\Lc(\ftight,R)$. 

\begin{proposition} \label{prop:diag.alg.isom}
The diagonal algebra $D(L_R(\G))$ is isomorphic to $\Lc(\ftight,R)$ and the diagonal C*-algebra $D(C^*(\G))$ is isomorphic to  $C_0(\ftight)$.
\end{proposition} 
    \begin{proof}
    The C*-algebra case is proved in \cite[Theorem 6.9]{MR3680957}. 
    
    We show that $D(L_R(\G))\cong \Lc(\ftight,R)$. Let $D$ be the $R$-algebra defined in (\ref{dfn:D.genrating.set}). By Proposition~\ref{prop:isom.ideals}, $\Lc(\ftight,R)$ is ismorphic to $D$ and, by Lemma~\ref{prop:iso.dual.D.with.tight}, this isomorphism sends  $1_{Y_{(a,A,a)}}\in D \mapsto 1_{V_{(a,A,a)}}\in \Lc(\ftight,R)$ for every $A\in \G^{(0)}$ and $\alpha\in \awstar$. Hence, 
    \[\Lc(\ftight,R)=\mathrm{span}\{1_{V_{(\alpha,A,\alpha)}}: \alpha\in \awstar,A\in \acf_{\alpha}\}.  \]
    Let $\kappa$ be the isomorphism given in Theorem~\ref{kart} and fix an $\alpha\in \awstar$ and an $A\in \acf$. Then, it follows from Equation~(\ref{eqn:isom.general.products}) that
    \[\kappa(s_{\alpha}p_{A}s_{\alpha^*})= 1_{\{\eword\}\times V_{(\alpha,A,\alpha)}}.\]
    That is, $1_{V_{(\alpha,A,\alpha)}}\in A_{R}(\{\eword\}\times \ftight) =A_R((\mathbb{F} \ltimes_{\Delta} \ftight)^{(0)})$.
    By identifying $(\mathbb{F} \ltimes_{\Delta} \ftight)^{(0)}$ with $\ftight$, it follows  that $1_{V_{(\alpha,A,\alpha)}}$ is a function in $\Lc(\ftight,R)$. Then, since $\kappa$ maps the generators of $D(L_R(\G))$ onto a generating set of $\Lc(\ftight,R)$, we have that $D(L_R(\G))\cong \Lc(\ftight,R)$. 
    \end{proof}


\begin{proposicao}\label{viajar}
 Let $R$ be a unital commutative ring. The abelian core $M(L_R(\G))$ is isomorphic to $A_R(\text{Iso}(\mathbb{F} \ltimes_{\Delta} \ftight)^0)$, and the abelian core  $M(C^*(\G))$ is isomorphic to $ C^*(\text{Iso}(\mathbb{F} \ltimes_{\Delta} \ftight)^0)$.
\end{proposicao}
\begin{proof}

We first show that $M(L_R(\G))\cong A_R(\text{Iso}(\mathbb{F} \ltimes_{\Delta} \ftight)^0)$. Let $\kappa$ denote the isomorphism in Theorem~\ref{kart}. By Proposition~\ref{prop:diag.alg.isom}, we have that $\kappa(D(L_R(\G)))\subset A_R(\text{Iso}(\mathbb{F} \ltimes_{\Delta} \ftight)^0)$. Let $\alpha,\beta\in\awstar$ and $A\in \acf$ such that $r(\alpha) \cap r(\beta) \cap A\neq \emptyset$ and $\alpha=\beta\gamma$ for some  loop $\gamma\in \awstar$ without exits. Notice in this case that $r(\alpha)=\{s(\gamma)\}$ so that $r(\alpha)\cap r(\beta)\cap A=\{s(\gamma)\}$.
Then, by Lemma~\ref{lem:isolated.vs.exits}, $\xi^{\beta\gamma^\infty}$ is an isolated point and 
\begin{equation} \label{eqn:singleton}
    V_{(\alpha, r(\alpha) \cap r(\beta) \cap A, \alpha)}=V_{(\beta\gamma^n, r(\alpha) \cap r(\beta) \cap A, \beta\gamma^n)}=\{\xi\}, 
\end{equation}
for all $n\geq 0$.
Therefore, by (\ref{eqn:isom.general.products}) and (\ref{eqn:singleton}), we have that   
    \begin{equation}
     \begin{split}
      \kappa(s_{\alpha}p_A s_{\beta^*})&= 1_{\{\alpha\beta^{-1}\}\times V_{(\alpha,r(\alpha)\cap r(\beta)\cap A,\alpha)}}  =1_{\{(\alpha\beta^{-1},\xi)\}}, \text{ and } \\
      \kappa(s_{\beta}p_A s_{\alpha^*})&= 1_{\{\beta\alpha^{-1}\}\times V_{(\beta,r(\alpha)\cap r(\beta)\cap A,\beta)}} =1_{\{(\beta\alpha^{-1},\xi)\}} 
     \end{split}
    \end{equation} 
Hence, $\kappa(M(L_R(\G)))\subseteq A_R(\text{Iso}(\mathbb{F} \ltimes_{\Delta} \ftight)^0)$. 

To see that $A_R(\text{Iso}(\mathbb{F} \ltimes_{\Delta} \ftight)^0)\subseteq \kappa(M(L_R(\G)))$, let $U\scj \text{Iso}(\mathbb{F} \ltimes_{\Delta} \ftight)^0$ be a compact-open bisection. Then, by Lemma~\ref{lem:interior.isotropy}, 
\[U=V\cup\{(t_1,\xi_1),\ldots,(t_m,\xi_m)\},\] 
where $V$ is a compact-open subset of $\Gamma^{(0)}$ and $(t_i,\xi_i)\in \text{Iso}(\mathbb{F} \ltimes_{\Delta} \ftight)^0\setminus (\mathbb{F} \ltimes_{\Delta} \ftight)^{(0)}$ is an isolated point.  By Proposition~\ref{prop:diag.alg.isom}, we have that $\kappa^{-1}(V)\subset M(L_R(\G))$. By Lemma~\ref{lem:interior.isotropy}, for each $i=1,\ldots,n$, the labelled path associated to $\xi_i$ is of the form $\beta_i\gamma_i^{\infty}$ for some loop $\gamma_i$ without exists and   $t_i=\beta_i\gamma_i\beta_i^{-1}$ or $t_i=\beta_i \gamma_i^{-1}\beta_i^{-1}$. Put $\alpha_i=\beta_i\gamma_i$. Then $r(\alpha_i)\cap r(\beta_i)=\{s(\gamma_i)\}\neq \emptyset$. Hence, $s_{\alpha_i}p_{r(\alpha_i)}s_{\beta_i^*},s_{\beta_i}p_{r(\alpha_i)}s_{\alpha_i^*}\in  M(L_R(\G))$. Now, if $t_i=\beta_i\gamma_i\beta_i^{-1}$, then 
\[      \kappa(s_{\alpha_i}p_{r(\alpha_i)}s_{\beta_i^*})= 1_{\{\alpha_i\beta_i^{-1}\}\times V_{(\alpha_i,r(\alpha_i)\cap r(\beta_i),\alpha_i)}}  =1_{\{(\alpha_i\beta_i^{-1},\xi_i)\}}=1_{\{(t_i,\xi_i)\}}, \]
and if $t_i=\beta_i\gamma_i^{-1}\beta_i^{-1}$, then 
 \[     \kappa(s_{\beta_i}p_{r(\alpha_i)}s_{\alpha_i^*})= 1_{\{\beta_i\alpha_i^{-1}\}\times V_{(\beta_i,r(\alpha_i)\cap r(\beta_i),\beta_i)}} =1_{\{(\beta_i\alpha_i^{-1},\xi_i)\}}=1_{\{(t_i,\xi_i)\}}.\] 
Thus, $\kappa^{-1}(A_R(\text{Iso}(\mathbb{F} \ltimes_{\Delta} \ftight)^0))\subseteq M(L_R(\G))$, which completes the proof that $M(L_R(\G))\cong A_R(\text{Iso}(\mathbb{F} \ltimes_{\Delta} \ftight)^0)$.

Next we show that $M(C^*(\G))$ is isomorphic to $ C^*(\text{Iso}(\mathbb{F} \ltimes_{\Delta} \ftight)^0)$. Let $\kappa_*$ be the isomorphism given in (\ref{eqn:isom.cstar.generators}). Then, the same arguments as for the algebraic case above imply that $\kappa_*$ maps a dense *-subalgebra of $M(C^*(\G))$ onto a dense *-subalgebra of $C^*(\text{Iso}(\mathbb{F}\rtimes_{\Delta}\ftight)^0)$. Since $\kappa_*$ is continuous, we may extend it to a *-isomorphism of $M(C^*(\G))$ onto $C^*(\text{Iso}(\mathbb{F}\rtimes_{\Delta}\ftight)^0)$.

\end{proof}

As a first consequence of the above proposition we obtain a generalized uniqueness theorem for ultragraph C*-algebras (this generalizes the graph C*-algebra version given in \cite[Theorem~3.13]{saragab}).

\begin{teorema}\label{pizza}
For a *-homomorphism $\pi:C^*(\G)\rightarrow A$, the  following conditions are
equivalent:
\begin{enumerate}
    \item $\pi$ is injective;
    \item the restriction of $\pi$ to $M(C^*(\G))$ is injective;
    \item $\pi(p_A)\neq 0$ for all nonempty $A\in \G^0$ and, for any simple loop $\alpha$ without exist, the spectrum of $\pi(s_\alpha)$ contains the unit circle.
\end{enumerate}
\end{teorema}
\begin{proof}
That 1. is equivalent to 2. follows from Theorem~\ref{thm:uniqueness} and Proposition~\ref{viajar}. The equivalence between 1 and 3 is given in \cite[Theorem~7.4]{MR3554458}.
\end{proof}

\begin{remark}
A version of the above theorem has been proved for higher rank graphs in \cite{BROWN20142590}.
\end{remark}

We now extend the generalized uniqueness theorem for Leavitt path algebras (\cite[Theorem~5.2]{GILCANTO2018227} to ultragraph Leavitt path algebras.

 \begin{teorema}\label{uniqueultra} Let $\G$ be an ultragraph and $R$ be a unital commutative ring. Consider $\Phi: L_R(\G) \rightarrow \mathcal{A}$ a ring homomorphism. Then  $\Phi$ is injective if, and only if,  the restriction of $\Phi$ to $M(L_R(\G))$ is injective.
\end{teorema}
\begin{proof}
This follows from Theorem~\ref{thm:Uniqueness} and Proposition~\ref{viajar}.
\end{proof}

We finish the paper applying the the results of \cite{HazratLi} to obtain an extension of \cite[Theorem~4.13]{GILCANTO2018227} to ultragraph Leavitt path algebra, and to describe when the core of an ultragraph Leavitt path algebra is equal to the center of the algebra.

\begin{corolario}
Let $\G$ be an ultragraph and $L_R(\G)$ be the ultragraph Leavitt path algebra associated to $\G$. Then 

\begin{enumerate}[\upshape(i)]

\item The centraliser of the diagonal algebra $D(L_R(\G))$ is the core algebra $M(L_R(\G))$. 

\item The core algebra $M(L_R(\G))$ is a maximal commutative subalgebra of $L_R(\G)$. 

\item If $Z(L_R(\G))= M(L_R(\G))$ and $\G$ is connected, then $L_R(\G)$ is either $R$ or $R[x,x^{-1}]$, i.e, the ultragraph $\G$ is either a single vertex or a vertex and an edge.

\end{enumerate}
\end{corolario}
\begin{proof}
(i) It follows from Proposition \ref{prop:diag.alg.isom} and \cite[Theorem 2.2]{HazratLi}.

(ii) It follows from Proposition \ref{viajar} and \cite[Corollary 2.3]{HazratLi}.

(iii) By \cite[Corollary 2.3]{HazratLi}, since the center of $L_R(\G)$ is commutative, we have that $L_R(G)$ itself is commutative. We claim that $\G$ has only one vertex. Indeed, if $v,w\in\G$ are such that $v\neq w$, then there exists a path $\alpha\in\G^*$ such that $s(\alpha)=v$ and $w\in r(\alpha)$. On one hand, we get $p_ws_{\alpha}=0$. On the other hand, $(\alpha,\{w\},\alpha)\in E(S)\setminus\{0\}$, so that, $s_{\alpha}p_w\neq 0$ by \eqref{eqn:isom.general.products} and Remark~\ref{rmk:Ve.not.empty}. Hence $p_ws_{\alpha}\neq s_{\alpha}p_w$, contradicting the commutativity of $L_R(\G)$. We conclude that $G^0$ is a singleton. Now suppose that there are two different edges $a,b\in\G^1$. In this case $ab$ and $ba$ are two different paths on $\G$, which also correspond to two different elements in $\mathbb{F}$. Again, using \eqref{eqn:isom.general.products} and Remark~\ref{rmk:Ve.not.empty}, we conclude that $s_as_b\neq s_bs_a$ contradicting the commutativity of $L_R(\G)$. Therefore, either $\G$ consists of only one vertex and no edges, in which case $L_R(\G)$ is $R$, or $\G$ consists of only a vertex and an edge, in which case $L_R(\G)$ is $R[x,x^{-1}]$.
\end{proof}

\bibliographystyle{abbrv}
\bibliography{ref}

\begin{thebibliography}{10}

\bibitem{MR2045419}
F.~Abadie.
\newblock On partial actions and groupoids.
\newblock {\em Proc. Amer. Math. Soc.}, 132(4):1037--1047, 2004.

\bibitem{MR3614028}
T.~Bates, T.~M. Carlsen, and D.~Pask.
\newblock {$C^*$}-algebras of labelled graphs {III}---{$K$}-theory
  computations.
\newblock {\em Ergodic Theory Dynam. Systems}, 37(2):337--368, 2017.

\bibitem{MR2304922}
T.~Bates and D.~Pask.
\newblock {$C\sp *$}-algebras of labelled graphs.
\newblock {\em J. Operator Theory}, 57(1):207--226, 2007.

\bibitem{MR3743184}
V.~M. Beuter and D.~Gon\c{c}alves.
\newblock The interplay between {S}teinberg algebras and skew rings.
\newblock {\em J. Algebra}, 497:337--362, 2018.

\bibitem{MR3648984}
G.~Boava, G.~G. de~Castro, and F.~de~L.~Mortari.
\newblock Inverse semigroups associated with labelled spaces and their tight
  spectra.
\newblock {\em Semigroup Forum}, 94(3):582--609, 2017.

\bibitem{MR3680957}
G.~Boava, G.~G. de~Castro, and F.~de~L.~Mortari.
\newblock {${\rm C}^*$}-algebras of labelled spaces and their diagonal {${\rm
  C}^*$}-subalgebras.
\newblock {\em J. Math. Anal. Appl.}, 456(1):69--98, 2017.

\bibitem{Gil3}
G.~Boava, G.~G. de~Castro, and F.~de~L.~Mortari.
\newblock Groupoid {M}odels for the {C}*-{A}lgebra of {L}abelled {S}paces.
\newblock {\em Bull. Braz. Math. Soc. (N.S.)}, 51(3):835--861, 2020.

\bibitem{BROWN20142590}
J.~H. Brown, G.~Nagy, and S.~Reznikoff.
\newblock A generalized {C}untz–{K}rieger uniqueness theorem for higher-rank
  graphs.
\newblock {\em J. Funct. Anal. J}, 266(4):2590 -- 2609, 2014.

\bibitem{BNRSW}
J.~H. Brown, G.~Nagy, S.~Reznikoff, A.~Sims, and D.~P. Williams.
\newblock Cartan subalgebras in ${C}^*$-algebras of {H}ausdorff étale
  groupoids.
\newblock {\em Integral Equations Operator Theory}, 85(1):109--126, 2016.

\bibitem{BCW}
N.~Brownlowe, T.~M. Carlsen, and M.~Whittaker.
\newblock Graph algebras and orbit equivalence.
\newblock {\em Ergodic Theory Dynam. Systems}, 37:389--417, 2017.

\bibitem{GILCANTO2018227}
C.~G. Canto and A.~Nasr-Isfahani.
\newblock The commutative core of a {L}eavitt path algebra.
\newblock {\em J. Algebra}, 511:227--248, 2018.

\bibitem{CEM}
L.~O. Clark and C.~Edie-Michell.
\newblock Uniqueness theorems for {S}teinberg algebras.
\newblock {\em Algebr. Represent. Theory}, 18, 2015.

\bibitem{CEP}
L.~O. Clark, R.~Exel, and E.~Pardo.
\newblock A generalized uniqueness theorem and the graded ideal structure of
  {S}teinberg algebras.
\newblock {\em Forum Math.}, 30(3):533--552, 2018.

\bibitem{center}
L.~O. Clark, D.~Martín~Barquero, C.~Martín~González, and M.~Siles~Molina.
\newblock Using the {S}teinberg algebra model to determine the center of any
  {L}eavitt path algebra.
\newblock {\em Israel J. Math.}, pages 23--44, 2019.

\bibitem{MR3856223}
G.~G. de~Castro and D.~Gon\c{c}alves.
\newblock K{MS} and ground states on ultragraph {${\rm C}^*$}-algebras.
\newblock {\em Integral Equations Operator Theory}, 90(6):Art. 63, 23, 2018.

\bibitem{GDD}
G.~G. de~Castro, D.~Gon\c{c}alves, and D.~W. van Wyk.
\newblock Topological full groups of ultragraph groupoids as an isormorphism
  invariant.
\newblock {\em Münster J. Math.}, to appear.

\bibitem{MR4109095}
G.~G. de~Castro and D.~W. van Wyk.
\newblock Labelled space {$C^*$}-algebras as partial crossed products and a
  simplicity characterization.
\newblock {\em J. Math. Anal. Appl.}, 491(1):124290, 35, 2020.

\bibitem{Deaconu95}
V.~Deaconu.
\newblock Groupoids associated with endomorphisms.
\newblock {\em Trans. Amer. Math. Soc.}, 347(5):1779--1786, 1995.

\bibitem{MR2419901}
R.~Exel.
\newblock Inverse semigroups and combinatorial {$C\sp \ast$}-algebras.
\newblock {\em Bull. Braz. Math. Soc. (N.S.)}, 39(2):191--313, 2008.

\bibitem{MR3699795}
R.~Exel.
\newblock {\em Partial dynamical systems, {F}ell bundles and applications},
  volume 224 of {\em Mathematical Surveys and Monographs}.
\newblock American Mathematical Society, Providence, RI, 2017.

\bibitem{Firrisa}
M.~M. Firrisa.
\newblock Morita equivalence of graph and ultragraph {L}eavitt path algebras.
\newblock {\em arXiv:2006.06521 [math.RA]}, 2020.

\bibitem{reduction}
D.~Gonçalves and D.~Royer.
\newblock Representations and the reduction theorem for ultragraph {L}eavitt
  path algebras.
\newblock {\em arXiv:1902.00013 [math.RA]}, 2019.

\bibitem{goncalves_royer_2019}
D.~Gonçalves and D.~Royer.
\newblock Simplicity and chain conditions for ultragraph {L}eavitt path
  algebras via partial skew group ring theory.
\newblock {\em Journal of the Australian Mathematical Society}, page 1–21,
  2019.

\bibitem{MR3554458}
D.~Gon\c{c}alves, H.~Li, and D.~Royer.
\newblock Branching systems and general {C}untz-{K}rieger uniqueness theorem
  for ultragraph {$C^*$}-algebras.
\newblock {\em Internat. J. Math.}, 27(10):1650083, 26, 2016.

\bibitem{grskew}
D.~Gon\c{c}alves and D.~Royer.
\newblock Leavitt path algebras as partial skew group rings.
\newblock {\em Comm. Algebra}, 42:127--143, 2014.

\bibitem{MR3600124}
D.~Gon\c{c}alves and D.~Royer.
\newblock Ultragraphs and shift spaces over infinite alphabets.
\newblock {\em Bull. Sci. Math.}, 141(1):25--45, 2017.

\bibitem{MR3938320}
D.~Gon\c{c}alves and D.~Royer.
\newblock Infinite alphabet edge shift spaces via ultragraphs and their {$\rm
  C^*$}-algebras.
\newblock {\em Int. Math. Res. Not. IMRN}, 2019(7):2177--2203, 2019.

\bibitem{GRirred}
D.~Gon\c{c}alves and D.~Royer.
\newblock Irreducible and permutative representations of ultragraph {L}eavitt
  path algebras.
\newblock {\em Forum Math.}, 32:417--431, 2020.

\bibitem{SoboG}
D.~Gon\c{c}alves and M.~Sobottka.
\newblock Continuous shift commuting maps between ultragraph shift spaces.
\newblock {\em Discrete Contin. Dyn. Syst.}, 39:1033--1048, 2019.

\bibitem{gonccalves2018li}
D.~Gon\c{c}alves and B.~B. Uggioni.
\newblock Li-{Y}orke chaos for ultragraph shift spaces.
\newblock {\em Discrete Contin. Dyn. Syst.}, 40(4):2347--2365, 2020.

\bibitem{gonccalves2019ultragraph}
D.~Gon\c{c}alves and B.~B. Uggioni.
\newblock Ultragraph shift spaces and chaos.
\newblock {\em Bull. Sci. Math.}, 158:102807, 23, 2020.

\bibitem{Hazgraded}
R.~Hazrat and H.~Li.
\newblock Graded {S}teinberg algebras and partial actions.
\newblock {\em J. Pure Appl. Algebra}, 222(8):3946--3967, 2018.

\bibitem{HazratLi}
R.~Hazrat and H.~Li.
\newblock A note on the centralizer of a subalgebra of {S}teinberg algebra.
\newblock {\em arXiv preprint arXiv:1912.01932}, 2020.

\bibitem{HazNam}
R.~Hazrat and T.~G. Nam.
\newblock Realizing ultragraph {L}eavitt path algebras as {S}teinberg algebras.
\newblock {\em arXiv:2008.04668 [math.RA]}, 2020.

\bibitem{imanfar2017leavitt}
M.~Imanfar, A.~Pourabbas, and H.~Larki.
\newblock The {L}eavitt path algebras of ultragraphs.
\newblock {\em Kyungpook Math. J.}, 60(1):21--43, 2020.

\bibitem{MR2413313}
T.~Katsura, P.~S. Muhly, A.~Sims, and M.~Tomforde.
\newblock Ultragraph {$C^*$}-algebras via topological quivers.
\newblock {\em Studia Math.}, 187(2):137--155, 2008.

\bibitem{Larki}
H.~Larki.
\newblock Primitive ideals and pure infiniteness of ultragraph
  ${C}^*$-algebras.
\newblock {\em J. Korean Math. Soc.}, 56(1):1--23, 2019.

\bibitem{MR2974110}
M.~V. Lawson.
\newblock Non-commutative {S}tone duality: inverse semigroups, topological
  groupoids and {$C^\ast$}-algebras.
\newblock {\em Internat. J. Algebra Comput.}, 22(6):1250058, 47, 2012.

\bibitem{MR2457327}
A.~E. Marrero and P.~S. Muhly.
\newblock Groupoid and inverse semigroup presentations of ultragraph
  {$C^*$}-algebras.
\newblock {\em Semigroup Forum}, 77(3):399--422, 2008.

\bibitem{saragab}
G.~Nagy and S.~Reznikoff.
\newblock {Abelian core of graph algebras}.
\newblock {\em J. Lond. Math. Soc}, 85(3):889--908, 03 2012.

\bibitem{Nam}
T.~G. Nam and N.~D. Nam.
\newblock Purely infinite simple ultragraph {L}eavitt path algebras.
\newblock {\em arXiv:2007.08144 [math.RA]}, 2020.

\bibitem{RenaultBook}
J.~Renault.
\newblock {\em A groupoid approach to {$C^{\ast} $}-algebras}, volume 793 of
  {\em Lecture Notes in Mathematics}.
\newblock Springer, Berlin, 1980.

\bibitem{Renault00}
J.~Renault.
\newblock Cuntz-like algebras.
\newblock In {\em Operator theoretical methods ({T}imi\c soara, 1998)}, pages
  371--386. Theta Found., Bucharest, 2000.

\bibitem{BenGroupoid}
B.~Steinberg.
\newblock A groupoid approach to discrete inverse semigroup algebras.
\newblock {\em Adv. Math.}, 223(2):689--727, 2010.

\bibitem{Ben}
B.~Steinberg.
\newblock Ideals of \'etale groupoid algebras and {E}xel's {E}ffros-{H}ahn
  conjecture.
\newblock {\em arxiv:1810.10580 [math.RA]}, 2018.

\bibitem{tasca2020kms}
F.~A. Tasca and D.~Gon{\c{c}}alves.
\newblock {KMS} states and continuous orbit equivalence for ultragraph shift
  spaces with sinks.
\newblock {\em arXiv preprint arXiv:2003.05793}, 2020.

\bibitem{MR2050134}
M.~Tomforde.
\newblock A unified approach to {E}xel-{L}aca algebras and {$C\sp
  \ast$}-algebras associated to graphs.
\newblock {\em J. Operator Theory}, 50(2):345--368, 2003.

\end{thebibliography}

Gilles Gon\c{c}alves de Castro, Departamento de Matem\'atica, Universidade Federal de Santa Catarina, Florian\'opolis, 88040-900, Brazil.

Email: gilles.castro@ufsc.br	

\vspace{0.5pc}

Daniel Gon\c{c}alves, Departamento de Matem\'atica, Universidade Federal de Santa Catarina, Florian\'opolis, 88040-900, Brazil.

Email: daemig@gmail.com

\vspace{0.5pc}

Daniel W van Wyk, Department of Mathematics, Dartmouth College, Hanover, NH 03755-3551, USA.

Email: daniel.w.van.wyk@dartmouth.edu

\end{document}